\documentclass[twoside,11pt]{article}

\usepackage[preprint]{jmlr2e}
\usepackage[T1]{fontenc}
\usepackage{amsmath,amssymb,mathtools}
\usepackage{graphicx}
\usepackage{subfigure}
\usepackage{algorithmic}
\usepackage{minitoc}
\usepackage[ruled,vlined]{algorithm2e}

\newtheorem{assumption}{Assumption}

\usepackage{lastpage}
\ShortHeadings{Entropic Fictitious Play}{Chen, Ren and Wang}
\firstpageno{1}

\input{mathcommands.cfg}

\begin{document}

\dosecttoc
\faketableofcontents

\title{Entropic Fictitious Play\\for Mean Field Optimization Problem}

\author{\name Fan Chen \email alexnwish@sjtu.edu.cn \\
\addr School of Mathematical Sciences, Shanghai Jiao Tong University \\
Shanghai, China
\AND
\name Zhenjie Ren \email ren@ceremade.dauphine.fr \\
\addr CEREMADE, Université Paris-Dauphine, PSL \\
Paris, France
\AND
\name Songbo Wang \email songbo.wang@polytechnique.edu \\
\addr CMAP, CNRS, École polytechnique, Institut Polytechnique de Paris \\
Palaiseau, France
}

\maketitle

\begin{abstract}%
We study two-layer neural networks in the mean field limit,
where the number of neurons tends to infinity.
In this regime,
the optimization over the neuron parameters becomes the optimization
over the probability measures,
and by adding an entropic regularizer,
the minimizer of the problem is identified as a fixed point.
We propose a novel training algorithm named entropic fictitious play,
inspired by the classical fictitious play in game theory
for learning Nash equilibriums, to recover this fixed point,
and the algorithm exhibits a two-loop iteration structure.
Exponential convergence is proved in this paper
and we also verify our theoretical results by simple numerical examples.
\end{abstract}


\section{Introduction}
\label{section:introduction}

Deep learning has achieved unprecedented success in numerous practical scenarios,
including computer vision, natural language processing
and even autonomous driving,
which leverages deep reinforcement learning techniques
\citep{Krizhevsky2012ImageNetCW,Goldberg2016APO, Arulkumaran2017DeepRL}.
Stochastic gradient algorithms (SGD) and their variants have been widely used
to train neural networks, that is, to minimize networks' loss
and thereby to fit the data available effectively
\citep{LeCun1998GradientbasedLA, Kingma2015AdamAM}.
However, due to the complicated network structures
and the non-convexity of typical optimization objectives,
mathematical guarantees of convergence to the optimizer remain elusive.
Recent studies on the insensibility of the number of neurons on one layer
when it is sufficiently large \citep{Hastie2022SurprisesIH},
and the feasibility of interchanging the neurons on one layer
\citep{Nguyen2020ARF, rotskoff2018neural}
both motivated the investigation of mean field regime.
In practice, over-parameterized neural networks
with a large number of neurons
are commonly employed in order to achieve high performance
\citep{Huang2017DenselyCC}.
This further motivates researchers to view neurons as random variables
following a probability distribution
and the summation over neurons
as an expectation with respect to this distribution
\citep{Sirignano2020MeanFA}.

Another appealing approach to address the global convergence
of over-parameterized networks is
through the \emph{neural tangent kernel} (NTK) regime \citep{Jacot2018NeuralTK}.
In this regime,
it is believed that
when the network width tends to infinity,
the parameter updates, driven by stochastic gradient descent,
do not significantly deviate
from i.i.d Gaussian initialization,
and these updates are called lazy training \citep{Tzen2020AMT, Chizat2019OnLT}.
As a result, training of neural networks can be depicted as regression with a fixed kernel given by linearization at initialization,
leading to the exponential convergence \citep{Jacot2018NeuralTK}.
By appropriate time rescaling,
it is possible for the dynamics of the kernel method
to track the SGD dynamics closely
\citep{mei2019mean, AllenZhu2019LearningAG}.
Other studies, such as \cite{Dou2020TrainingNN},
explore the reproducing kernel Hilbert space
and demonstrate that the gradient flow indeed converges
to the kernel ridgeless regression with an adaptive kernel.
Besides in \cite{Chen2020AGN}, the researchers extend the definition of the kernel
and show that the training with an appropriate regularizer
also exhibits behaviors similar to the kernel method.
However, the kernel behavior primarily manifests
during the early stages of the training process,
whereas the mean field model reveals and explains
the longer-term characteristics \citep{mei2019mean}.
Furthermore, another advantage of the mean field settings
compared to NTK is the presence of feature learning,
in contrast to the perspective of random feature
\citep{Suzuki2019AdaptivityOD, Ghorbani2019LimitationsOL}.

In the mean field limit where neurons become infinitely many,
the dynamics of the neuron parameters under gradient descent
can be understood as a gradient flow of measures in Wasserstein-\(2\) space,
providing a geometric interpretation of the learning algorithm.
This flow is also described by a PDE system
where the unknown is the density function of the measure.
Well-posedness of the PDE system, discretization errors
and finite-time propagation of chaos
are studied in recent works
\citep{Nguyen2020ARF, mei2019mean, pmlr-v134-fang21a,Araujo2019AML,Sirignano2022MeanFA}.
On the other hand, extensive analysis has been conducted
to investigate the convergence of such dynamics to their equilibrium.
The convergence of gradient flows modeling shallow networks
is studied in \cite{chizat2018global, mei2019mean, HRSS19};
more recent works extend the gradient-flow formulation
and study deep network structures \citep{pmlr-v134-fang21a,Nguyen2020ARF}.
Sufficient conditions for the convergence under non-convex loss functions
have been given in \cite{Nguyen2020ARF},
and the discriminatory properties of the non-linear activation function
have been exploited in
\cite{Sirignano2022MeanFA, rotskoff2018neural} to deduce the convergence.

In this paper, one key assumption is the convexity of the objective functional
with respect to its measure-valued argument.
This assumption has been exploited by many recent works.
Notably, \cite{Nitanda2022ConvexAO}
have established the exponential convergence of the entropy-regularized problem
in both discrete and continuous-time settings
by utilizing the log-Sobolev inequality (LSI),
following the observations in \cite{Nitanda2021ParticleDA}.
Additionally, \cite{Nitanda2017StochasticPG} estimate the generalization error
and prove a polynomial convergence rate by leveraging quadratic expansions
of the loss function.
\cite{Wei2019RegularizationMG} also prove polynomial convergence rates
in different scenarios,
where they add noise to the gradient descent and
assume the activation and regularization functions are homogeneous.

With the existing convergence results on gradient flows
for the mean field optimization problem in mind,
the following question arises to us:
\begin{quote}
\centering
\textit{Do there exist dynamics
other than gradient flows\\
that solve the (regularized) mean field optimization efficiently?}
\end{quote}
We believe the quest for its answer will not be wasted efforts,
as it may lead to potentially highly performant algorithms
for training neural networks,
and also because the dynamics similar to that we consider in this paper
have already found applications to various mean field problems.

We recall the classical fictitious play in game theory
originally introduced by Brown \cite{Brown51} to learn Nash equilibriums.
During the fictitious play,
in each round of repeated games,
each player optimally responds to the empirical frequency of actions
taken by their opponents (hence the name).
While the fictitious play does not necessarily converge in general cases
\citep{Shapley64},
it does converge for zero-sum games \citep{Robinson51}
and potential games \citep{MS96}.
More recently, this method has been revisited
in the context of mean field games
\citep{CH17,HS19,PPLGEP20,lavigne2022generalized}.

In this paper, we draw inspiration from the classical fictitious play
and propose a similar algorithm,
called \emph{entropic fictitious play} (EFP),
to solve mean field optimization problems
emerging from the training of two-layer neural networks.
Our algorithm shares a two-loop iteration structure
with the particle dual average (PDA) algorithm,
recently proposed by \cite{Nitanda2021ParticleDA}.
They estimated the computational complexity
and conducted various numerical experiments for PDA
to show its effectiveness in solving regularized mean field problems.
However, PDA is essentially different
from our EFP algorithm and their differences will be discussed in
Sections~\ref{sec:problemsetting} and \ref{sec:numerical}.

\section{Problem Setting}
\label{sec:problemsetting}

Let us first recall how the (convex) mean field optimization problem emerges
from the training of two-layer neural networks.
While the universal representation theorem
tells us that a two-layer network can arbitrarily well approximate
the continuous function on the compact time interval
\citep{cybenko1989approximation,barron1993universal},
it does not tell us how to find the optimal parameters.
One is faced with the non-convex optimization problem
\begin{equation}
\label{intro:min0}
\min_{\beta_{n,i}\in\mathbb R, \alpha_{n,i}\in \mathbb R^{d}, \gamma_{n,i}\in \mathbb R}
\int_{\mathbb R\times\mathbb R^{d}}
\ell\biggl( y , \frac1n\sum_{i=1}^n \beta_{n,i} \varphi (\alpha_{n,i}\cdot z + \gamma_{n,i}) \biggr)
\,\nu(dy\,dz),
\end{equation}
where \(\theta \mapsto \ell(y, \theta)\) is convex for every \(y\),
\(\varphi:\mathbb R\rightarrow\mathbb R\)
is a bounded, continuous and non-constant activation function,
and \(\nu\) is a measure of compact support representing the data.
Denote the empirical law of the parameters \(m^n\) by
\(m^n = \frac1n\sum_{i=1}^{n} \delta_{( \beta_{n,i},
\alpha_{n,i}, \gamma_{n,i} )}\).
Then the neural network output can be written by
\[
\frac1n\sum_{i=1}^n \beta_{n,i} \varphi(\alpha_{n,i}\cdot z+\gamma_{n,i})
= \int_{\mathbb R^{d+2}} \beta \varphi(\alpha\cdot z +\gamma )\,m^n(d\beta\,d\alpha\,d\gamma).
\]
For technical reasons we may introduce a truncation function \(h(\cdot)\)
whose parameter is denoted by \(\beta\) as in \cite{HRSS19}.
To ease the notation we denote \(x=(\beta,\alpha, \gamma)\in \mathbb R^{d+2}\)
and \(\hat \varphi(x,z) = h(\beta)\varphi(\alpha \cdot z+\gamma)\).
Denote also by \(\mathbb E^m = \mathbb E^{X\sim m}\)
the expectation of the random variable \(X\) of law \(m\).
Now we relax the original problem \eqref{intro:min0}
and study the mean field optimization problem
over the probability measures,
\begin{equation}\label{eq:intro_F_nn}
\min_{m\in \mathcal{P}(\mathbb{R}^d)} F(m),
\quad\text{where}~
F(m)\coloneqq \int_{\mathbb R^d} \ell\bigr( y , \mathbb E^m[ \hat \varphi(X,z)] \bigl)
\,\nu(dy\,dz)
\end{equation}
This reformulation is crucial, because the \emph{potential} functional \(F\)
defined above is convex in the space of probability measure.
In this paper, as in \cite{HRSS19,mei2018mean},
we shall add a relative entropy term
\(H(m|g)\coloneqq \int_{x \in \mathbb R^d} \log \frac{dm}{dg} (x)\,m(dx)\)
in order to regularize the problem.
The regularized problem then reads
\begin{equation}\label{intro:min}
\min_{m\in \mathcal{P}(\mathbb{R}^d)} V^\sigma(m),
\quad\text{where}~V^\sigma(m)\coloneqq F(m) + \frac{\sigma^2}{2} H(m|g).
\end{equation}
Here we choose the probability measure \(g\)
to be a Gibbs measure with energy function \(U\),
that is, the density of \(g\) satisfies \(g(x) \propto \exp\bigl(-U(x)\bigr)\).
It is worth noting that if a probability measure has finite entropy
relative to the Gibbs measure \(g\),
then it is absolutely continuous with respect to the Lebesgue measure.
Hence the density of \(m\) exists whenever \(V^\sigma (m)\) is finite.
In the following, we will abuse the notation
and use the same letter to denote the density function of \(m\).

Since \(F\) is convex, together with mild conditions,
the first-order condition says that
\(m^*\) is a minimizer of \(V^\sigma\) if and only if
\begin{equation}\label{eq:intro-foc}
\frac{\delta F}{\delta m} (m^*, x) + \frac{\sigma^2}{2} \log m^*(x) + \frac{\sigma^2}{2} U(x)
= \text{constant},
\end{equation}
where \(\frac{\delta F}{\delta m}\) is the linear derivative,
whose definition is postponed to
Assumption~\ref{assum:wellposedness} below.
Further, note that \(m^*\) satisfying \eqref{eq:intro-foc} must be
an invariant measure to the so-called mean field Langevin (MFL) diffusion:
\[
dX_t = -\biggl(\nabla_x \frac{\delta F}{\delta m} (m_t, X_t)
+ \frac{\sigma^2}{2} \nabla_x U(X_t) \biggr)\,dt + \sigma\,dW_t,
\quad\text{where}~m_t \coloneqq \Law(X_t).
\]
In \cite{HRSS19} it has been shown that
the MFL marginal law \(m_t\) converges towards \(m^*\),
and this provides an algorithm to approximate the minimizer \(m^*\).

The starting point of our new algorithm is to view the first-order condition
\eqref{eq:intro-foc} as a fixed pointed problem.
Given \(m \in \mathcal{P}(\mathbb{R}^d)\),
let \(\Phi(m)\) be the probability measure such that
\begin{equation}\label{eq:intro-fixedpoint}
\frac{\delta F}{\delta m} (m, x) + \frac{\sigma^2}{2} \log \Phi(m)(x) + \frac{\sigma^2}{2} U(x) = \text{constant}.
\end{equation}
By definition, a probability measure
\(m\) satisfies the first-order condition \eqref{eq:intro-foc}
if and only if \(m\) is a fixed point of \(\Phi\).
Throughout the paper we shall assume that there exists at most
one probability measure satisfying the first-order condition
(equivalently, there exists at most one fixed point for \(\Phi\)).
This is true when the objective functional \(F\) is convex.
Indeed, as the relative entropy \(m\mapsto H(m|g)\) is strictly convex,
the free energy
\(V^\sigma = F + \frac{\sigma^2}{2}H(\cdot|g)\)
is also strictly convex and therefore admits at most one minimizer.

It remains to construct an algorithm to find the fixed point.
Observe that \(\Phi(m)\) defined in \eqref{eq:intro-fixedpoint}
satisfies formally
\begin{equation}
\label{eq:intro-variation}
\Phi(m) = \argmin_{\mu\in \mathcal P(\mathbb{R}^d)}
\mathbb E^{X\sim \mu} \biggl[\frac{\delta F}{\delta m} (m, X) \biggr]
+ \frac{\sigma^2}{2} H(\mu|g),
\end{equation}
that is, the mapping \(\Phi\) is given by the solution to a variational problem,
similar to the definition of Nash equilibrium.
This suggests that we can adapt the classical fictitious play algorithm
to approach the minimizer.
In this context, \(\Phi(m_t)\) is the ``best response'' to \(m_t\)
in the sense of \eqref{eq:intro-variation},
and we define the evolution of the ``empirical frequency''
of the player's actions by
\begin{equation}\label{eq:fictitious-intro}
d m_t = \alpha\,\bigl(\Phi(m_t) - m_t \bigr)\,dt,
\end{equation}
where \(\alpha\) is a positive constant
and should be understood as the learning rate.
The Duhamel's formula for this equation reads
\[
m_t = \int_0^t \alpha e^{-\alpha (t - s)}\,\Phi (m_s)\,ds + e^{-\alpha t}\,m_0,
\]
so \(m_t\) is indeed a \emph{weighted} empirical frequency
of the previous actions \(m_0\) and \(\bigl(\Phi(m_s)\bigr)_{s\leq t}\).

We propose a numerical scheme corresponding to the entropic fictitious play
described informally in Algorithm~\ref{alg:efp-informal},
which consists of inner and outer iterations.
The inner iteration,
described later in Algorithm~\ref{alg:efp-formal} for a specific example,
calculates an approximation of \(\Phi(m_t)\) given the measure \(m_t\).
Note that we are sampling a classical Gibbs measure
so various Monte Carlo methods can be used.
The outer iterations let the measure evolve following
the entropic fictitious play \eqref{eq:fictitious-intro}
with a chosen time step \(\Delta t\).

\begin{algorithm}
\label{alg:efp-informal}
\caption{Entropic fictitious play algorithm}
\LinesNumbered
\KwIn{objective functional \(F\),
reference measure \(g \propto \exp (-U)\),
initial distribution \(m_0\),
time step \(\Delta t\),
interation times \(T\).}
\For{\rm \(t = 0\), \(\Delta t\), \(2\Delta t\), \(\ldots\,\),
\(T - \Delta t\)}{
\tcp{Inner iteration}
Sample \(\Phi(m_{t+\Delta t}) \propto
\exp \bigl(-\frac{\delta F}{\delta m}(m_t,x) -\frac{\sigma^2}{2} U(x) \bigr)\)
by Monte Carlo; \\
\tcp{Outer iteration}
Update \(m_{t+\Delta t} \leftarrow (1 - \alpha \Delta t)\,m_t
+ \alpha \Delta t\,\Phi(m_n)\); \\
}
\KwOut{distribution \(m_T\).}
\end{algorithm}

\subsection{Related Works}

\subsubsection{Mean Field Optimization}

In contrast to the entropy-regularized mean field optimization
addressed by our EFP algorithm,
the unregularized optimization has also been studied in recent works
\citep{chizat2018global,rotskoff2018neural,Sirignano2022MeanFA}.
\cite{pmlr-v134-fang21a} developed a mean field framework
that captures the feature evolution during multi-layer networks' training
and analyze the global convergence
for fully-connected neural networks and residual networks,
introduced by \cite{he2016deep}.
Deep network settings have also been studied in
\cite{Sirignano2022MeanFA,nguyen2019mean,
Araujo2019AML,pham2021global,Nguyen2020ARF}.

\subsubsection{Exponential Convergence Rate}

The exponential convergence rate of the mean field Langevin dynamics
has been shown in \cite{Nitanda2022ConvexAO}
by exploiting the log-Sobolev inequality,
which critically relies on the non-vanishing entropic regularization.
On the other hand, \cite{Chizat2022MeanFieldLD} has studied
the annealed mean field Langevin dynamics,
where the time steps decay following an \(O\bigl((\log t)^{-1}\bigr)\) trend,
and has shown the convergence
towards the minimizer of the unregularized objective functional.
In this paper, we will also prove an exponential convergence rate
for our EFP algorithm
and the precise statement can be found in Theorem~\ref{thm:exp}.
The convergence rate obtained solely depends on the learning rate,
which can be chosen in a fairly arbitrary way.
This seems to be an improvement over the LSI-dependent rate in
\cite{Nitanda2022ConvexAO, Chizat2022MeanFieldLD}.
However, the arbitrariness is due to the fact that
our theoretical result only addresses
the outer iteration and assumes
that the target measure of inner one can be perfectly sampled
(see Algorithm~\ref{alg:efp-informal}),
and our convergence rate can not be directly compared
to the ones obtained by \cite{Nitanda2022ConvexAO, Chizat2022MeanFieldLD}.
However, the inner iteration aims to sample a Gibbs measure,
which is a classical task
for which various Monte Carlo algorithms are available.
(see Remark~\ref{rem:gibbs sampling}).
Furthermore, we propose a ``warm start'' technique
to alleviate the computational burden of the inner iterations
(see Algorithm~\ref{alg:efp-formal}).

\subsubsection{Particle Dual Averaging}
Our entropic fictitious play algorithm
shares similarities with the particle dual averaging algorithm
introduced in \cite{Nitanda2021ParticleDA}.
PDA is an extension of
regularized dual average studied in
\cite{Nesterov2005SmoothMO,Xiao2009DualAM},
and can be considered the particle version
of the dual averaging method designed
to solve the regularized mean field optimization problem \eqref{intro:min}.
The key feature shared by PDA and EFP
is the two-loop iteration structure.
In the PDA outer iteration, we calculate a moving average \(\tilde f_n\)
of the linear functional derivative
of the objective \(\frac{\delta F}{\delta m}\),
\begin{equation}
\label{eq:PDA-formula}
\tilde f_{n} = (1-\alpha \Delta t)\,
\frac{\delta F}{\delta m}(\tilde m_{n-1},\cdot)
+ \alpha \Delta t\,\frac{\delta F}{\delta m}(\tilde m_{n-1},\cdot);
\end{equation}
the measure \(\tilde m_n\) is on the other hand updated by the inner iteration,
\begin{equation}
\label{eq:PDA-inner}
\tilde m_{n+1}(x) = \argmin_{m \in \mathcal{P}(\mathbb{R}^d)}
\mathbb{E}^m\bigl[\tilde f_{n}(x)\bigr] +\frac{\sigma^2}{2} H(m|g),
\end{equation}
which can be calculated by a Gibbs sampler.
While the PDA inner iteration \eqref{eq:PDA-inner}
is identical to that of EFP,
their outer iterations are distinctly different.
The PDA outer iteration updates the linear derivatives
\(\frac{\delta F}{\delta m} (\tilde m_n, \cdot)\)
by forming a convex combination,
while the EFP outer iteration updates
the measures by a convex combination,
which serves as the first argument
of the linear derivative \(\frac{\delta F}{\delta m}(\cdot,\cdot)\).
One disadvantage of PDA is that
one needs to store the history of measures
\((\tilde m_i^{\vphantom n})_{i=1}^n\) to evaluate \(\tilde f_{n}\) in \eqref{eq:PDA-formula},
which may lead to high memory usage in numerical simulations.
Our EFP algorithm circumvents this numerical difficulty
as the dynamics \eqref{eq:fictitious-intro}
corresponds to a birth-death particle system whose memory usage is bounded
(see discussions in Section~\ref{sec:simulation-efp}).
As a side note, EFP and PDA coincide
when the mapping \(m\mapsto \frac{\delta F}{\delta m}(m, \cdot)\) is linear.
This occurs when \(F\) is quadratic in \(m\).
For example, if \(F\) is defined by \eqref{eq:intro_F_nn}
with a quadratic loss, \(\ell(y,\theta) = |y-\theta|^2\),
then its functional derivative
\[
\frac{\delta F}{\delta m}(m, x)
= 2 \int_{\mathbb R^d} \bigl( \mathbb E^m[ \hat \varphi(X,z)]
- y \bigr) \hat \varphi(x,z)\,\nu(dy\,dz)
\]
is linear in \(m\).
Another difference is that the PDA outer iteration is updated
with diminishing time steps (or equivalently, learning rates)
\(\Delta t = O \bigl(n^{-1}\bigr)\),
which leads to the absence of exponential convergence,
while EFP fixes the time step \(\Delta t\)
and exhibits exponential convergence
(modulo the errors from the inner iterations).
Finally, the condition (A3) of \cite{Nitanda2021ParticleDA}
seems difficult to verify and our method does not rely on such an assumption.

\subsection{Organization of Paper}

In Section~\ref{sec:mainresult} we state our results
on the existence and convergence of entropic fictitious play.
In Section~\ref{sec:numerical} we provide a toy numerical experiment
to showcase the feasibility of the algorithm
for the training two-layer neural networks.
Finally the proofs are given in Section~\ref{sec:proof}
and they are organized in several subsections
with a table of contents in the beginning to ease the reading.

\section{Main Results}
\label{sec:mainresult}

Fix an integer \(d > 0\) and a real number \(p \geq 1\).
Denote by \(\mathcal P (\mathbb R^d)\) the set of the probability measures
on \(\mathbb R^d\)
and by \(\mathcal P_p(\mathbb R^d)\) the set of those with finite \(p\)-moment.
We suppose the following assumption throughout the paper.

\begin{assumption}
\label{assum:wellposedness}
\begin{enumerate}
\item
The mean field functional
\(F : \mathcal P (\mathbb R^d) \to \mathbb R\) is non-negative
and \(C^1\), that is, there exists a continuous function,
also called \emph{functional linear derivative},
\(\frac{\delta F}{\delta m} : \mathcal P (\mathbb R^d) \times \mathbb R^d
\to \mathbb R\)
such that for every \(m_0\), \(m_1 \in \mathcal P (\mathbb R^d)\),
\[
F(m_1) - F(m_0)
= \int_0^1 \int_{\mathbb R^d} \frac{\delta F}{\delta m} (m_\lambda, x)
(m_1 - m_0)\,m_\lambda(dx)\,d\lambda,
\]
where \(m_\lambda \coloneqq (1 - \lambda) m_0 + \lambda m_1\).
Moreover, there exists constants
\(L_F\), \(M_F > 0\) such that
for every \(m\), \(m' \in \mathcal P (\mathbb R^d)\)
and for every \(x\), \(x' \in \mathbb R^d\),
\begin{align}
\label{eq:lipschitz-delta-F-delta-m}
\biggl| \frac{\delta F}{\delta m} (m, x) - \frac{\delta F}{\delta m} (m', x') \biggr|
&\leq L_F \bigl(\mathcal W_p (m,m') + \lvert x - x'\rvert\bigr), \\
\label{eq:boundedness-delta-F-delta-m}
\biggl| \frac{\delta F}{\delta m} (m, x) \biggr| &\leq M_F.
\end{align}

\item
The function \(U : \mathbb R^d \to \mathbb R\) is measurable and satisfies
\[
\int_{\mathbb R^d} \exp \bigl(-U(x)\bigr)\,dx = 1.
\]
Moreover it satisfies
\[
\essinf_{x \in \mathbb R^d} U(x) > -\infty \quad \text{and}\quad
\liminf_{x \to \infty} \frac{U (x)}{\lvert x\rvert^p} > 0.
\]
\end{enumerate}
\end{assumption}

Given a function \(U\) satisfying Assumption~\ref{assum:wellposedness},
define the Gibbs measure \(g\) on \(\mathbb R^d\) by its density
\(g(x) \coloneqq \exp ( - U(x))\).
In particular, given \(m \in \mathcal P_p (\mathbb R^d)\),
we can consider the relative entropy between \(m\) and \(g\),
\footnote{
The relative entropy is defined to be \(+\infty\)
whenever the integral is not well defined.
Therefore, the relative entropy is defined
for every measure in \(\mathcal P(\mathbb R^d)\)
and is always non-negative.}
\[
H(m|g)
= \int_{x \in \mathbb R^d} \log \frac{dm}{dg} (x)\,m(dx).
\]

In this paper we consider the entropy-regularized optimization
\[
\inf_{m \in \mathcal P (\mathbb R^d)} V^\sigma(m),
\quad\text{where}~V^\sigma(m) \coloneqq F(m) + \frac{\sigma^2}{2} H(m|g).
\]
Our aim is to propose a dynamics of probability measures
converging to the minimizer of the value function \(V^\sigma\).

\begin{proposition}
\label{prop:existence-of-minimizer}
If Assumption~\ref{assum:wellposedness} holds,
then there exists at least one minimizer of \(V^\sigma\),
which is absolutely continuous with respect to the Lebesgue measure
and belongs to \(\mathcal{P}_p (\mathbb{R}^d)\).
\end{proposition}

Given the result above, we can restrict ourselves
to the space of probability measures of finite \(p\)-moments
when we look for minimizers of the regularized problem \(V^\sigma\).
Before introducing the dynamics,
let us recall the first-order condition for being a minimizer.

\begin{proposition}[Proposition 2.5 of \citealt{HRSS19}]\label{prop:foc}
Suppose Assumption~\ref{assum:wellposedness} holds.
If \(m^*\) minimizes \(V^\sigma\) in \(\mathcal P(\mathbb R^d)\),
then it satisfies the \emph{first-order condition}
\begin{equation}
\label{eq:foc}
\frac{\delta F}{\delta m}(m^*,\cdot) + \frac{\sigma^2}{2}\log m^*(\cdot)
+ \frac{\sigma^2}{2}U(\cdot)
~\text{is a constant Leb-a.e.,}
\end{equation}
where \(m^*(\cdot)\) denotes the density function of the measure \(m^*\).

Conversely, if \(F\) is additionally convex,
then every \(m^*\) satisfying \eqref{eq:foc} is a minimizer of \(V^\sigma\)
and such a measure is unique.
\end{proposition}

\begin{definition}
For each \(\mu \in \mathcal P ( \mathbb R^d)\),
define \(G(\mu; \cdot) : \mathcal P (\mathbb R^d) \to \mathbb R\) by
\begin{equation}
G ( \mu ; m) = \expect^{X \sim \mu}
\biggl[ \frac{\delta F}{\delta m} (m, X) \biggr].
\end{equation}
Furthermore, given \(m\in \mathcal P (\mathbb R^d)\), we define
a measure \(\hat m \in \mathcal P(\mathbb R^d)\) by
\begin{equation}\label{defn:hatm}
\hat m = \argmin_{\mu\in \mathcal P (\mathbb R^d)}
G ( \mu ; m)+ \frac{\sigma^2}{2} H(\mu|g),
\end{equation}
whenever the minimizer exists and is unique.
\end{definition}

\begin{proposition}
\label{prop:existence-uniqueness-of-m-hat-in-Pp}
Suppose Assumption~\ref{assum:wellposedness} holds.
The minimizer defined in \eqref{defn:hatm} exists, is unique,
and belongs to \(\mathcal P_p (\mathbb R^d)\).
This defines a mapping
\(\mathcal P_p (\mathbb R^d) \ni m
\mapsto \hat m \in \mathcal P_p (\mathbb R^d)\),
which we denote by \(\Phi\) in the following.
\end{proposition}

Since
\(\frac{\delta G}{\delta \mu}(\mu, x; m) = \frac{\delta F}{\delta m} (m, x)\),
according to the first-order condition in Proposition~\ref{prop:foc},
\(\hat m\) must satisfy
\begin{equation}
\label{eq:foc-m-hat}
\frac{\delta F}{\delta m}(m,\cdot) + \frac{\sigma^2}{2}\log\hat m + \frac{\sigma^2}{2}U~\text{ is a constant Leb-a.e.}
\end{equation}
Therefore, a probability measure \(m\) is a fixed point of the mapping \(\Phi\)
if and only if it satisfies the first-order condition \eqref{eq:foc}.
In particular,
by Propositions~\ref{prop:existence-of-minimizer} and \ref{prop:foc},
there exists at least one minimizer of \(V^\sigma\),
and it is a fixed point of the mapping \(\Phi\).
On the other hand, if \(\Phi\) admits only one fixed point,
then it must be the unique minimizer of \(V^\sigma\).

Given the definition of \(\hat m\),
the \emph{entropic fictitious play} dynamics is the flow of measures
\((m_t)_{t \geq 0}\) defined by
\begin{equation}
\label{eq:efp}
\frac{dm_t}{dt} = \alpha\,(\hat m_t - m_t).
\end{equation}
This equation is understood in the sense of distributions a priori.
We shall show that the entropic fictitious play converges
towards the minimizer of \(V^\sigma\) under mild conditions.

\begin{remark}
Choosing the relative entropy to be the regularizer may seem arbitrary.
It is motivated by the following two observations:
\begin{itemize}
\item If \(F\) is convex, the strict convexity of entropy ensures
that the mapping \(\Phi\) admits at most one fixed point.
\item In numerical applications,
one needs to sample the distribution \(\hat m_t\) efficiently.
Applying the entropic regularization,
we can sample \(\hat m_t\) by Monte Carlo methods since
it is in the form of a Gibbs measure according to \eqref{defn:hatm}.
See Section~\ref{sec:numerical} for more details.
\end{itemize}
\end{remark}

\begin{definition}[Dynamical system per Definition 4.1.1
of \citealt{henry2006geometric}]
Let \(S[t]\) be a mapping from \(\mathcal W_p\) to itself
for every \(t \geq 0\).
We say the collection \(( S[t] )_{t \geq 0}\)
is a \emph{dynamical system} on \(\mathcal W_p\) if
\begin{enumerate}
\item \(S[0]\) is the identity on \(\mathcal W_p\);
\item \(S[t] (S[t'] m) = S[t+t'] m\)
for every \(m \in \mathcal{P}_p(\mathbb R^d)\) and \(t\), \(t' \geq 0\);
\item for every \(m \in \mathcal{P}_p(\mathbb R^d)\),
\(t\mapsto S[t] m\) is continuous;
\item for every \(t \geq 0\), \(m\mapsto S[t] m\) is continuous
with respect to the topology of \(\mathcal W_p\).
\end{enumerate}
\end{definition}

\begin{proposition}[Existence and wellposedness of the dynamics]
\label{prop:wellposedness-time-dynamics}
Suppose Assumption~\ref{assum:wellposedness} holds.
Let \(\alpha\) be a positive real
and let \(m_0\) be in \(\mathcal P_p (\mathbb R^d)\) for some \(p \geq 1\).
Then there exists a solution
\((m_t)_{t\geq 0} \in C\bigl( [0, +\infty); \mathcal W_p \bigr)\)
to \eqref{eq:efp}.

When \(p = 1\), the solution is unique and
depends continuously on the initial condition.
In other words, there exists a dynamical system \((S[t])_{t\geq 0}\)
on \(\mathcal W_1\)
such that \(m_t\) defined by \(m_t = S[t]m_0\)
solves \eqref{eq:efp}.

If additionally the initial value \(m_0\) is absolutely continuous
with respect to the Lebesgue measure,
then the solution \(m_t\) admits density for every \(t > 0\),
and the densities \(m_t (\cdot)\)
solves \eqref{eq:efp} classically.
That is to say, for every \(x \in \mathbb R^d\)
the mapping \(t \mapsto m_t (x)\) is \(C^1\) on \([0, +\infty)\)
and the derivative satisfies
\begin{equation}
\label{eq:efp-pointwise}
\frac{\partial m_t (x)}{\partial t} = \alpha \bigl(\hat m_t (x) - m_t(x) \bigr).
\end{equation}
for every \(t > 0\).
\end{proposition}

Now we study the convergence of the entropic fictitious play dynamics
and to this end we introduce the following assumption.

\begin{assumption}
\label{assum:convergence}
\begin{enumerate}
\item The mapping
\(\Phi : \mathcal P_p (\mathbb R^d) \ni m \mapsto \hat m \in \mathcal P_p (\mathbb R^d)\)
admits a unique fixed point \(m^*\).

\item The initial value \(m_0\) belongs to
\(\mathcal P_{p'} (\mathbb R^d)\) for some \(p' > p\)
and \(H(m_0 | g) < +\infty\).
\end{enumerate}
\end{assumption}

\begin{remark}
Under Assumption~\ref{assum:wellposedness},
the first condition above is implied the convexity of \(F\).
Indeed, if \(F\) is convex,
then the regularized objective \(V^\sigma\)
reads \(V^\sigma = F + H(\cdot|g)\)
and is therefore strictly convex.
So it admits a unique minimizer \(m^*\) in \(\mathcal P_p(\mathbb R^d)\)
and by our previous arguments \(m^*\)is also the unique fixed point
of the mapping \(\Phi\).
\end{remark}

\begin{theorem}[Convergence in the general case]
\label{thm:general-convergence}
Let Assumptions~\ref{assum:wellposedness} and \ref{assum:convergence} hold.
If \((m_t)_{t \geq 0}\) is a flow of measures in \(\mathcal W_p\)
solving \eqref{eq:efp},
then \(m_t\) converges to \(m^*\) in \(\mathcal W_p\) when \(t \to +\infty\),
and for every \(x \in \mathbb R^d\),
\(m_t (x) \to m^*(x)\) when \(t \to +\infty\).

Moreover, the mapping \(t \mapsto V^\sigma(m_t)\) is differentiable with
derivative
\[
\frac{dV^\sigma (m_t)}{dt}
= - \frac{\alpha \sigma^2}{2} \bigl( H(m_t | \hat m_t) + H(\hat m_t | m_t) \bigr),
\]
and it satisfies
\[
\lim_{t \to +\infty} V^\sigma (m_t) = V^\sigma (m^*).
\]
\end{theorem}

Given the convexity and higher differentiability of \(F\),
we also show that the convergence of \(V^\sigma(m_t)\) is exponential.

\begin{assumption}
\label{assum:convex}
The mean-field function \(F\) is convex and \(C^2\) with bounded derivatives.
That is to say, there exists a continuous and bounded function
\(\frac{\delta^2 F}{\delta m^2} :
\mathcal P (\mathbb R^d) \times \mathbb R^d \times \mathbb R^d \to \mathbb R\)
such that it is the linear functional derivative
of \(\frac{\delta F}{\delta m}\).
\end{assumption}

\begin{theorem}
\label{thm:exp}
Let Assumptions~\ref{assum:wellposedness},
\ref{assum:convergence} and \ref{assum:convex} hold.
Then we have for every \(t \geq 0\),
\[
0\leq V^\sigma(m_t) - \inf_{m\in \mathcal P(\mathbb R^d)} V^\sigma(m)
\leq \frac{\sigma^2}{2} H(m_0|\hat m_0) e^{-\alpha t}.
\]
\end{theorem}

\section{Numerical Example}
\label{sec:numerical}

In this section we walk through the implementation
of the entropic fictitious play in details
by treating a toy example.
Recall that in Algorithm~\ref{alg:efp-informal} the measures are updated
following the outer iteration
\[
\frac{dm_t}{dt} = \alpha\,(\hat m_t - m_t),
\]
and \(\hat m_t = \Phi(m_t)\) is evaluated by the inner iteration.

\subsection{Evaluation of Gibbs measure}

Since \(\hat m_t\) is a Gibbs measure corresponding to the potential
\( \frac{\delta F}{\delta m}(m_t,\cdot) + \frac{\sigma^2}{2} U\),
it is the unique invariant measure of a Langevin dynamics
under the following technical assumptions on \(F\) and \(U\).

\begin{assumption}
\label{assum:langevin}
\begin{enumerate}
\item For all \(m \in \mathcal P (\mathbb R^d)\), the function
\(\frac{\delta F}{\delta m} (m, \cdot) : \mathbb R^d \to \mathbb R\)
has a locally Lipschitz derivative, i.e.
the intrinsic derivative of \(F\),
\(DF(m,\cdot) \coloneqq \nabla \frac{\delta F}{\delta m}(m, \cdot)\)
exists everywhere and is locally Lipschitz.

\item The function \(U\) is \(C^2\), and there exists \(\kappa > 0\)
such that \(\bigl(\nabla U(x) - \nabla U(y)\bigr) \cdot (x - y)
\geq \kappa (x - y)^2\) when \(\lvert x-y\rvert\) is sufficiently large.
\end{enumerate}
\end{assumption}

\begin{proposition}
\label{prop:langevin}
Suppose Assumptions~\ref{assum:wellposedness} and \ref{assum:langevin} hold.
Let \(m\) be a probability measure on \(\mathbb R^d\).
Then a probability measure \(\hat m \in \mathcal P (\mathbb R^d)\)
satisfies the condition \eqref{eq:foc-m-hat}
if and only if it is the unique stationary measure of the Langevin dynamics
\begin{equation}
\label{eq:langevin-m-hat}
d \Theta_s = - \biggl( DF(m,\Theta_s) + \frac{\sigma^2}{2} \nabla U(\Theta_s) \biggr)\,ds
+ \sigma\,dW_s,
\end{equation}
where \(W\) is a standard Brownian motion.
Moreover, if \(\Law(\Theta_0)\in \cup_{p>2}\mathcal{P}_p(\mathbb R^d)\),
then the marginal distributions \(\Law(\Theta_s)\) converge
in Wasserstein-\(2\) distance towards the invariant measure.
\end{proposition}

We refer readers to Theorem 2.11 of \cite{HRSS19}
for the proof of the proposition.

\begin{remark}
\label{rem:gibbs sampling}
\begin{enumerate}
\item Various Markov chain Monte Carlo (MCMC) methods
are available for sampling Gibbs measures
\citep{andrieu2003introduction, karras2022overview}.
Here in our inner iteration,
we simulate the Langevin diffusion \eqref{eq:langevin-m-hat}
by the simplest unadjusted Langevin algorithm (ULA) proposed in
\cite{parisi1981correlation}.
However, there are many other efficient MCMC methods for our aim.
For example, we could employ the Metropolis-adjusted Langevin algorithms
or the Hamiltonian Monte Carlo (HMC) methods based on an underdamped dynamics
with fictitious momentum variables \citep{neal2011mcmc}.

\item Exponential convergence in the sense of relative entropy
for ULA proposed above is shown in \cite{Vempala2019RapidCO},
based on a log-Sobolev inequality condition for potential.
There are also convergence results in the sense of
the Wasserstein and total variation distance for Langevin Monte Carlo.
For example,
\cite{Durmus2019HighdimensionalBI} prove Wasserstein convergence for ULA,
\cite{bou2021mixing,cheng2018underdamped} prove respectively
convergence in total variation and in Wasserstein distance
for Hamiltonian Monte Carlo.
\end{enumerate}
\end{remark}

\subsection{Simulation of Entropic Fictitious Play}
\label{sec:simulation-efp}

Now we explain our numerical scheme of the entropic fictitious play
dynamics \eqref{eq:efp}.
First we approximate the probability distributions \(m_t\)
by empirical measures of particles in the form
\[ m_t = \frac{1}{N}\sum_{i=1}^N \delta_{\Theta^i_t},\]
where \(\Theta^i_t \in \mathbb R^d\) encapsulates all the parameters
of a single neuron in the network.
In order to evaluate the Gibbs measure \(\hat m_t\),
we simulate a system of \(M\) Langevin particles
using the Euler scheme for a long enough time \(S\), i.e.
\begin{equation}
\label{eq:euler-langevin}
\Theta^i_{t,s+\Delta s} = \Theta^i_{t,s} - \biggl( DF(m_t,\Theta^i_{t,s})
+ \frac{\sigma^2}{2} \nabla U(\Theta^i_{t,s}) \biggr)\,\Delta s
+ \sigma \sqrt{\Delta s}\,\mathcal{N}^i_{t,s},
\end{equation}
for \(1\leq i\leq M\) and \(s < S\),
where \(\mathcal{N}^i_{t,s}\) are independent standard Gaussian variables.
We then set \(\hat m_t\) equal to the empirical measure of the particles
at the final time \(S\),
\((\Theta^i_{t,S})_{1\leq i\leq M}\), i.e.
\[
\hat m_t \coloneqq \frac1M\sum_{i=1}^M \delta_{\Theta^i_{t,S}}.
\]

To speed up the EFP inner iteration we adopt the following warm start technique.
For each \(t\),
the initial value of the inner iteration
\((\Theta^i_{t+\Delta t,0})_{1\leq i\leq M}\)
is chosen to be the final value of the previous inner iteration,
i.e. \((\Theta^i_{t,S})_{1\leq i\leq M}\).
This approach exploits the continuity of the mapping \(\Phi\)
proved in Corollary~\ref{corollary:wasserstein-lipschitz-phi}:
if \(\Phi\) is continuous,
the measures \(\Phi(m_{t + \Delta t})\), \(\Phi (m_t)\)
should be close to each other as long as
\(m_{t + \Delta t}\), \(m_t\) are close,
and this is expected to hold when the time step \(\Delta t\) is small.
Hence this choice of initial value for the inner iterations
should lead to less error in sampling the Gibbs measure \(\hat m_t\).

Then we explain how to simulate the outer iteration.
The na\"ive approach is to add particles to the empirical measures by
\[
m_{t + \Delta t}
= (1 - \alpha\Delta t)\,m_t + \alpha \Delta t\,\hat m_t
= \frac{1 - \alpha\Delta t}{N}\sum_{i=1}^N \delta_{\Theta^i_{t}}
+ \frac{\alpha \Delta t}{N}\sum_{i=1}^N \delta_{\Theta^i_{t,S}}.
\]
However, this leads to a linear explosion of the number of particles
when \(t \to +\infty\) as at each step it is incremented by \(M\).
To avoid this numerical difficulty,
we view the EFP dynamics \eqref{eq:efp} as a birth-death process
and kill \(\lfloor \alpha \Delta t N \rfloor\) particles
before adding the same number of particles that represents \(\hat m_t\),
calculated by the Gibbs sampler.
In this way, the number of particles to keep remains bounded uniformly in time
and the memory use never explodes.

\subsection{Training a Two-Layer Neural Network by Entropic Fictitious Play}

We consider the mean field formulation
of two-layer neural networks in Section~\ref{section:introduction}
with the following specifications.
We choose the loss function \(\ell\) to be quadratic:
\(\ell(y, \theta) = \frac 12 |y-\theta|^2\),
and the activation function to be the modified ReLU,
\(\varphi (t) =\max\bigl(\min (t,5),0\bigr)\).
We also fix a truncation function \(h\) defined by
\(h(x) = \max\bigl(\min (x,5),-5\bigr)\).
In this case, the objective functional \(F\) reads
\[
F(m) = \frac 1{2K} \sum_{k=1}^K \bigl( y_k
- \expect^m \bigl[h(\beta)\varphi(\alpha \cdot z_k + \gamma) \bigr] \bigr)^2.
\]
where \((\alpha, \beta, \gamma)\) is a random variable distributed as \(m\)
and \((z_k, y_k)_{k=1}^{K}\) is the data set
with \(z_k\) being the features and \(y_k\) being the labels.
Finally we choose the reference measure \(g\)
by fixing \(U(x) = \frac{1}{2} x^2 + \text{constant}\),
where the constant ensures that \(\int g = \int \exp\bigl(-U(x)\bigr)\,dx = 1\).
Under this choice, one can verify Assumptions~\ref{assum:wellposedness},
\ref{assum:convergence}, \ref{assum:convex},
and the Langevin dynamics \eqref{eq:euler-langevin}
for the inner iteration at time \(t\) reads
\[
\begin{aligned}
d \beta_s&=
\frac 1Kh'(\beta_s) \varphi(\alpha_s \cdot z_k+\gamma_s)
\sum_{k=1}^{K}
\bigl(y_k - \expect^{m_t}\bigl[h(\beta)\varphi(\alpha\cdot z_k+\gamma)\bigr]\bigr)
\,ds
- \frac{\sigma^2}{2}\,\beta_s\,ds+\sigma\,dW_s^\beta \\
d \alpha_s&=
\frac 1Kh(\beta_s) z_k \varphi'(\alpha_s \cdot z_k+\gamma_s)
\sum_{k=1}^{K}
\bigl(y_k - \expect^{m_t}\bigl[h(\beta)\varphi(\alpha\cdot z_k+\gamma)\bigr]\bigr)\,ds
- \frac{\sigma^2}{2}\,\alpha_s\,ds+\sigma\,dW_s^\alpha \\
d \gamma_s&=
\frac 1Kh(\beta_s) \varphi'(\alpha_s \cdot z_k+\gamma_s)
\sum_{k=1}^{K}
\bigl(y_k - \expect^{m_t}\bigl[h(\beta)\varphi(\alpha\cdot z_k+\gamma)\bigr]\bigr)\,ds
- \frac{\sigma^2}{2}\,\gamma_s\,ds+\sigma\,dW_s^\gamma
\end{aligned}
\]
where \(W^{\{\alpha,\beta,\gamma\}}\) are independent standard Brownian motions
in respective dimensions.
The discretized version of this dynamics
is then calculated on the interval \([0,S]\).

As a toy example,
we approximate the \(1\)-periodic sine function \(z \mapsto \sin(2\pi z)\)
defined on \([0,1]\) by a two-layer neural network.
We pick \(K=101\) samples evenly distributed on the interval \([0,1]\),
i.e. \(z_k = \frac{k-1}{101}\),
and set \(y_k = \sin 2 \pi z_k\) for \(k=1\), \(\ldots\,\), \(101\).
The parameters for the outer iteration are
\begin{itemize}
\item time step \(\Delta t = 0.2\),
\item horizon \(T=120.0\),
\item learning rate \(\alpha = 1\),
\item the number of neurons \(N=1000\),
\item the initial distribution of neurons \(m_0 = \mathcal N (0, 15^2)\).
\end{itemize}
For each \(t\),
we calculate the inner iteration \eqref{eq:euler-langevin}
with the parameters:
\begin{itemize}
\item regularization \(\sigma^2 / 2 = 0.0005\),
\item time step \(\Delta s = 0.1\),
\item time horizon for the first step \(S_\text{first}=100.0\),
and the remaining \(S_\text{other} = 5.0\),
\item the number of particles for simulating the Langevin dynamics
\(M = N = 1000\),
\end{itemize}
See Algorithm~\ref{alg:efp-formal} for a detailed description.

We present our numerical results.
We plot the learned approximative functions for different training epochs
(\(t / \Delta t = 10\), \(20\), \(50\), \(100\), \(200\), \(600\))
and compare them to the objective in Figure {\ref{subfig:function-value}}.
We find that in the last training epoch the sine function is well approximated.
We also investigate the validation error,
calculated from \(1000\) evenly distributed points in the interval \([0,1]\),
and plot its evolution in Figure~\ref{subfig:error}.
The final validation error is of the order of \(10^{-4}\)
and the whole training process consumes 63.02 seconds
on the laptop (CPU model: i7-9750H).
However, the validation error does not converge to \(0\),
possibly due to the entropic regularizer added to the original problem.

\begin{algorithm}
\label{alg:efp-formal}
\caption{EFP with Langevin inner iterations}
\LinesNumbered
\KwIn{objective function \(F(\cdot)\),
reference measure \(g\) with potential \(U\),
regularization parameter \(\sigma\),
initial distribution of parameter \(m_0\),
outer iterations time step \(\Delta t\) and horizon \(T\),
inner iterations time step \(\Delta s\) and horizon \(S\),
learning rate \(\alpha\),
and number of particles in simulation \(N\).}
generate i.i.d. \(\Theta_0^i \sim m_0\), \ \(i = 1\), \(\ldots\,\), \(N\); \\
\((\Theta_{0,0}^i)_{i=1}^N \leftarrow (\Theta_0^i)_{i=1}^N\); \\
\For{\rm \(t=0\), \(\Delta t\), \(2\Delta t\), \(\ldots\,\), \(T-\Delta t\)}{
\eIf{\(t = 0\)}
{\(S \leftarrow S_\text{first}\);}
{\(S \leftarrow S_\text{other}\);}
\tcp{Inner iterations}
\For{\rm \(s=0\), \(\Delta s\), \(2\Delta s\), \(\ldots\,\), \(S-\Delta s\)}{
generate standard normal variable \(\mathcal{N}^i_{t,s}\); \\
\tcp{Update the inner particles by Langevin dynamics}
\For{\rm\(i = 1\), \(2\), \(\ldots\,\), \(N\)}{
\(\Theta^i_{t,s+\Delta s} \leftarrow
\Theta^i_{t,s} - \bigl( DF(m_t,\Theta^i_{t})
+ \frac{\sigma^2}{2} \nabla U(\Theta^i_{t,s}) \bigr)\,\Delta s
+ \sigma \sqrt{\Delta s}\,\mathcal{N}^i_{t,s}\);
}
}
\tcp{Outer iteration}
\(K \leftarrow \lfloor \alpha \Delta t N \rfloor\); \\
choose uniformly \(K\) numbers from \(\{1, \ldots, N\}\)
and denote them by \((i_k)_{k=1}^K\); \\
\For{\rm \(i = 1\), \(2\), \(\ldots\,\), \(N\)}{
\eIf{\(i \in \{ i_k \}_{k=1}^K\)}{
\(\Theta^{i}_{t + \Delta t} \leftarrow \Theta^{i}_{t, S}\); \\
}
{
\(\Theta^i_{t + \Delta t} \leftarrow \Theta^i_t\); \\
}
}
\tcp{Warm start for inner iterations}
\For{\rm\(i=1\), \(2\), \(\ldots\,\), \(N\)}{
\(\Theta^i_{t+\Delta t,0} \leftarrow \Theta^i_{t,S}\); \\
}
}
\KwOut{distribution \(m_T = \frac1N\sum_{i=1}^N \delta_{\Theta^i_{T}}\).}
\end{algorithm}

\begin{figure}[htbp]
\subfigure[Approximated function value.]{
\begin{minipage}[t]{0.45\textwidth}
\centering
\includegraphics[width=\linewidth]{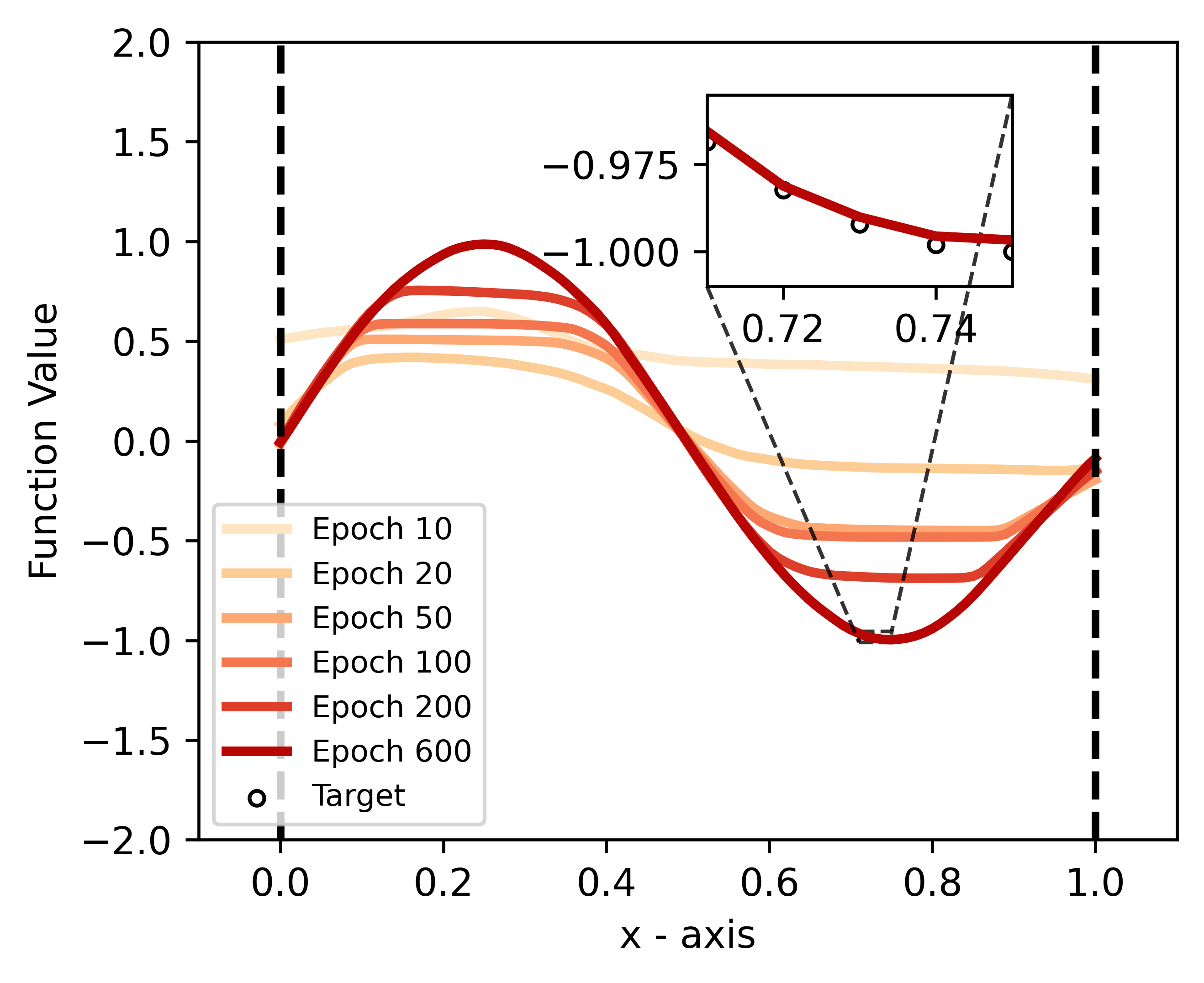}
\vfil
\label{subfig:function-value}
\end{minipage}
}
\hfil
\subfigure[Validation error in training.]{
\begin{minipage}[t]{0.45\textwidth}
\centering
\includegraphics[width=\linewidth]{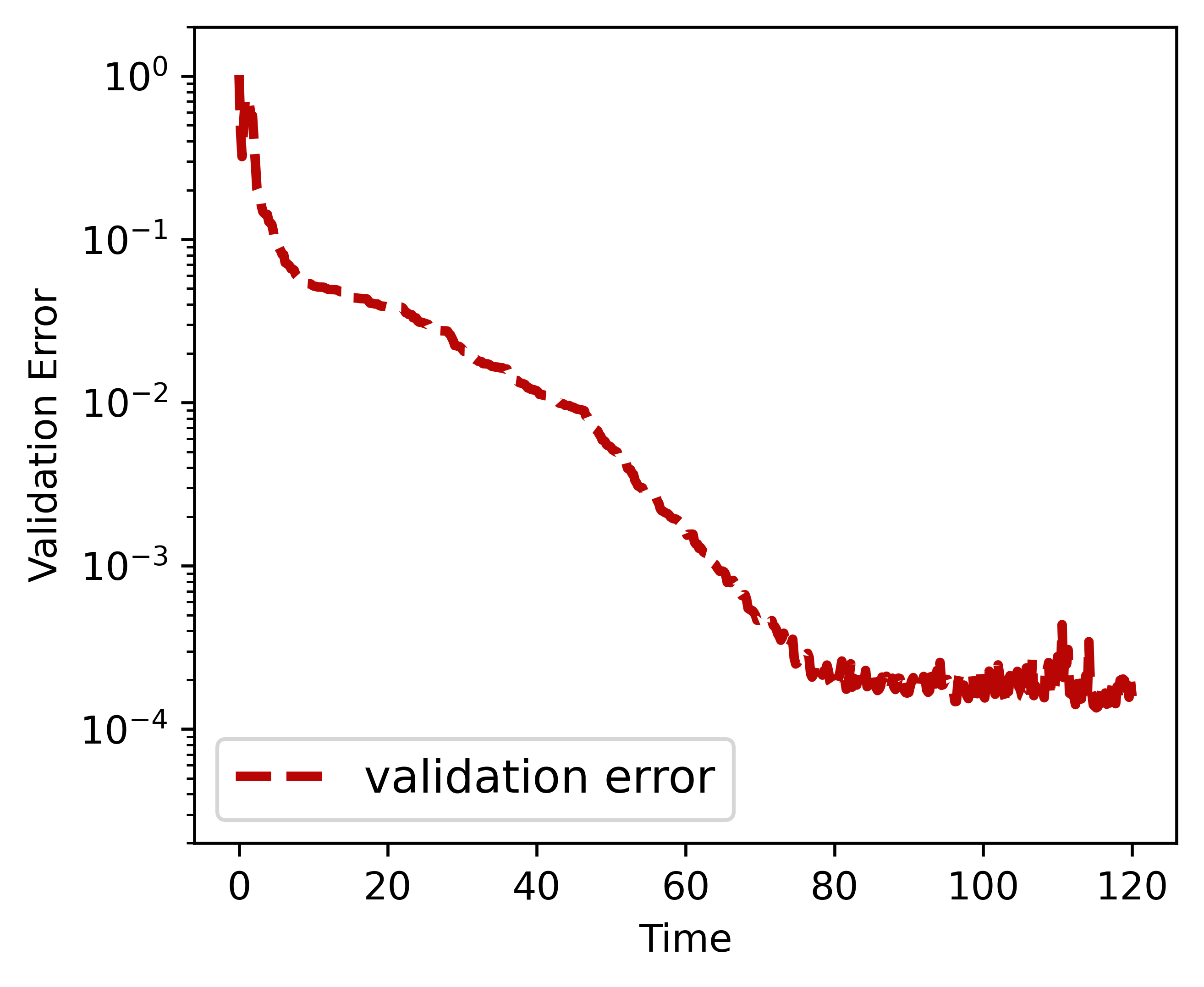}
\vfil
\label{subfig:error}
\end{minipage}
}
\hfil
\centering
\caption{(a) The approximated function value at different time: the colors from shallow to deep represents the number of outer iterations processed, epoch 10, 20, 50, 100, 200, 600 respectively; (b) The validation error at different training epochs.}
\end{figure}

\newpage

\section{Proofs}
\label{sec:proof}
\secttoc

\subsection{Proof of Propositions
\ref{prop:existence-of-minimizer} and
\ref{prop:existence-uniqueness-of-m-hat-in-Pp}}

\begin{proof}{\bf of Propositions
\ref{prop:existence-of-minimizer} and
\ref{prop:existence-uniqueness-of-m-hat-in-Pp}\ }
We only show Proposition~\ref{prop:existence-of-minimizer}
as the method is completely the same for the other proposition.

By Assumption~\ref{assum:wellposedness} we have
\(\liminf_{x\to\infty} U(x)/\lvert x\rvert^{p} > 0\).
Then we can find \(R\), \(c > 0\) such that
\(U(x) \geq c \lvert x\rvert^p\) for \(\lvert x\rvert > R\).
Choose a minimizing sequence \((m_n)_{n \in \mathbb N}\) in the sense that
\(V^\sigma (m_n) \searrow
\inf_{m \in \mathcal P (\mathbb R^d)} V^\sigma (m)\)
when \(n \to +\infty\).
Then we have
\begin{align*}
\sup_{n\in\mathbb N} V^\sigma (m_n) &\geq H(m_n | e^{-U} )
= \int m_n(x) \bigl( \log m_n (x) + U(x) \bigr)\,dx \\
&= \biggl(\int_{\lvert x\rvert \leq R} + \int_{\lvert x\rvert > R}\biggr) m_n (x) \bigl( \log m_n (x) + U(x) \bigr)\,dx \\
& \geq - \frac{c_d R^d}{e} + \essinf_{\lvert x\rvert \leq R} U(x)
+ \int_{\lvert x\rvert > R} m_n (x) \bigl( \log m_n (x) + U(x) \bigr)\,dx \\
& \geq - \frac{c_d R^d}{e} + \essinf_{\lvert x\rvert \leq R} U(x)
+ \int_{\lvert x\rvert > R} m_n (x) \bigl( \log m_n (x)
+ c \lvert x\rvert^p \bigr)\,dx,
\end{align*}
where the second inequality is due to \(x\log x \geq - e^{-1}\)
and \(c_d\) denotes the volume of the \(d\)-dimensional unit ball.

Define \(\tilde Z = \int_{ \lvert x\rvert > R}
\exp \bigl( - c |x|^p / 2\bigr)\,dx\)
and denote by \(\tilde g\) the probability measure
\[\tilde g (dx) = \frac{\mathbf 1_{|x| > R}}{\tilde Z}
\exp \biggl( - \frac c2 \lvert x\rvert^p \biggr)\,dx\]
supported on \(\{\lvert x\rvert > R\}\).
Using the fact that the relative entropy is always nonnegative, we have
\begin{align*}
&\hspace{-1em}\int_{\lvert x\rvert > R} m_n (x) \bigl( \log m_n (x) + c \lvert x\rvert^p \bigr)\,dx \\
&= \int_{\lvert x\rvert > R} m_n (x) \biggl( \log m_n (x) + \frac c2 \lvert x\rvert^p + \frac c2 \lvert x\rvert^p \biggr)\,dx \\
&= H (m_n | \tilde g ) - \log \tilde Z\int_{\lvert x\rvert > R} m_n (x)\,dx
+ \frac c2 \int_{\lvert x\rvert > R} m_n (x) \lvert x\rvert^p\,dx \\
&\geq - \lvert \log \tilde Z \rvert + \frac c2 \int_{\lvert x\rvert > R} m_n (x) \lvert x\rvert^p\,dx.
\end{align*}
Combining the two inequalities above, we obtain
\[
\frac c2 \int_{\lvert x\rvert > R} m_n (x) \lvert x\rvert^p\,dx
\leq \lvert \log \tilde Z \rvert + \frac{c_d R^d}{e}
- \essinf_{\lvert x\rvert < R} U(x) + \sup_{n \in \mathbb N} V^\sigma (m_n),
\]
which implies
\[
\sup_{n\in\mathbb N} \lVert m_n \rVert^p_p
= \sup_{n\in\mathbb N} \int m_n (x) \lvert x\rvert^p\,dx
< +\infty,
\]
that is, the \(p\)-moment of the minimizing sequence is uniformly bounded.
So the sequence \((m_n)_{n \in \mathbb N}\) is tight
and \(m_n \to m^*\) weakly for some \(m^* \in \mathcal P (\mathbb R^d)\)
along a subsequence.
Applying the following lemma, whose proof is postponed,
we obtain \(m^* \in \mathcal P_p (\mathbb R^d)\).

\begin{lemma}[``Fatou's lemma'' for weak convergence of measure]
\label{lemma:Fatou-for-weak-convergence}
Let \(X\) be a metric space,
\(f : X \to \mathbb R_+\) be nonnegative continuous function
and \((m_n)_{n \in \mathbb N}\) be a sequence of probability measures on \(X\).
If \(m_n\)
converges to another probability measure \(m\) weakly, then
\[
\int_{X} f\,dm \leq \liminf_{n\to+\infty} \int_X f\,dm_n.
\]
\end{lemma}

Since the relative entropy is weakly lower-semicontinuous,
the entropy of \(m^*\) satisfies
\[H ( m^* | g) \leq \liminf_{n \to +\infty} H(m_n | g).\]
We show the regular part satisfies \(\lim_{n\to+\infty} F(m_n) = F(m^*)\).
Indeed, by the definition of functional derivative, we have
\[
\bigl| F(m_n) - F(m^*) \bigr| \leq \int_0^1 \biggl| \int_{\mathbb R^d}
\frac{\delta F}{\delta m} ( m_{\lambda, n}, x )\,
(m_n - m) (dx) \biggr|\,d\lambda
\]
where \(m_{\lambda,n} \coloneqq (1 - \lambda) m_n + \lambda m\).
For every \(\lambda \in [0,1]\), we have
\begin{multline*}
\biggl| \int_{\mathbb R^d} \frac{\delta F}{\delta m}( m_{\lambda, n}, x )
\,(m_n - m^*) (dx) \biggr| \\
\leq \biggl| \int_{\mathbb R^d} \frac{\delta F}{\delta m} (m^*, x)
\,(m_n - m^* ) (dx) \biggr|
+ \int_{\mathbb R^d} \biggl| \frac{\delta F}{\delta m} (m_n, x) - \frac{\delta F}{\delta m} (m^*, x) \biggr|\,(m_n + m^* ) (dx).
\end{multline*}
Since \(\frac{\delta F}{\delta m} (m^*, \cdot)\) is a bounded continuous function, the weak convergence \( m_n \to m^*\) implies
\[\lim_{n\to+\infty} \biggl| \int_{\mathbb R^d} \frac{\delta F}{\delta m} (m^*, x)\,(m_n - m^* ) (dx) \biggr| = 0.\]
It remains to show the second term also converges to \(0\).
Since the convergence
\(\frac{\delta F}{\delta m} (m_n, x) \to \frac{\delta F}{\delta m} (m^*, x)\)
is uniform for \(\lvert x\rvert < R\) for every \(R > 0\),
we have
\[\lim_{n\to+\infty} \int_{\lvert x\rvert \leq R}
\biggl|\frac{\delta F}{\delta m} (m_n, x) - \frac{\delta F}{\delta m} (m^*, x)\biggr|
\,(m_n + m^*) (dx) = 0.\]
Consequently,
\begin{align*}
\limsup_{n\to+\infty} &\int_{\mathbb R^d} \biggl| \frac{\delta F}{\delta m} (m_n, x) - \frac{\delta F}{\delta m} (m^*, x) \biggr|\,(m_n + m^* ) (dx) \\
&= \limsup_{n\to+\infty} \biggl(\int_{\lvert x\rvert \leq R} + \int_{\lvert x\rvert > R}\biggr)
\biggl| \frac{\delta F}{\delta m} (m_n, x) - \frac{\delta F}{\delta m} (m^*, x) \biggr|\,(m_n + m^* ) (dx) \\
&\leq \limsup_{n\to+\infty} \int_{\lvert x\rvert > R} \biggl| \frac{\delta F}{\delta m} (m_n, x) - \frac{\delta F}{\delta m} (m^*, x) \biggr|\,(m_n + m^* ) (dx) \\
&\leq M_F \limsup_{n\to+\infty} \int_{\lvert x\rvert > R} (m_n + m^* ) (dx) \\
&= M_F \limsup_{n\to+\infty} \bigl(m^* ( \{ \lvert x\rvert > R \} )
+ m_n ( \{ \lvert x\rvert > R \} ) \bigr) = 0
\end{align*}
by tightness of the sequence \((m_n)_{n \in \mathbb N}\).
Finally, using the boundedness
\[\biggl| \int_{\mathbb R^d} \frac{\delta F}{\delta m} ( m_{\lambda, n}, x )\,(m_n - m) (dx) \biggr| \leq 2M_F,\]
we can apply the dominated convergence theorem
and show that when \(n \to +\infty\),
\[
\int_0^1 \biggl| \int_{\mathbb R^d} \frac{\delta F}{\delta m} ( m_{\lambda, n}, x )\,(m_n - m) (dx) \biggr|\,d\lambda \to 0.
\]

Summing up, we have obtained a measure
\(m^* \in \mathcal P_p (\mathbb R^d)\) such that
\begin{multline*}
V^\sigma (m^*) = F(m^*) + \frac{\sigma^2}{2} H(m^*) \\
\leq \liminf_{n\to+\infty} F(m_n) + \frac{\sigma^2}{2} H(m_n)
= \liminf_{n\to+\infty} V^\sigma (m_n) = \inf_{m \in \mathcal P(\mathbb R^d)} V^\sigma (m).
\end{multline*}
This completes the proof.
\end{proof}

\begin{proof}{\bf of Lemma~\ref{lemma:Fatou-for-weak-convergence}\ }
By the construction of Lebesgue integral, for every positive measure
\(\mu \in \mathcal P(X)\), we have
\[\int_X f\,d\mu = \sup_{M \geq 0} \int_X f \wedge M\,d\mu.\]
Therefore,
\begin{align*}
\int_X f\,dm &= \sup_{M \geq 0} \int_X f \wedge M\,dm
= \sup_{M \geq 0} \liminf_{n \to +\infty} \int_X f \wedge M\,dm_n
= \sup_{M \geq 0} \sup_{n} \inf_{k > n} \int_X f \wedge M\,dm_k \\
&\leq \sup_{n} \inf_{k > n} \sup_{M \geq 0} \int_X f \wedge M\,dm_k
= \liminf_{n \to +\infty} \sup_{M \geq 0 }\int_X f \wedge M\,dm_n
= \liminf_{n \to +\infty} \int_X f\,dm_n,
\end{align*}
where the inequality is due to \(\sup \inf \leq \inf \sup\).
\end{proof}

\subsection{Proof of Proposition~\ref{prop:wellposedness-time-dynamics}}

We prove several technical results before proceeding to the proof of
Proposition~\ref{prop:wellposedness-time-dynamics}.

\begin{proposition}
\label{prop:existence-uniqueness-regularity-m-hat}
Suppose Assumption~\ref{assum:wellposedness} holds.
For every \(m \in \mathcal P_p ( \mathbb R^d)\),
the measure \(\hat m\) determined by
\begin{equation}
\label{eq:definition-m-hat}
\hat m = \frac{1}{Z_m} \exp \biggl(- \frac{\delta F}{\delta m} (m, x) - U(x)\biggr),
\end{equation}
where \(Z_m\) is the normalization constant,
is well defined and belongs to \(\mathcal P_p (\mathbb R^d)\).
Moreover, there exists constants \(c\), \(C\) with \(0 < c < 1 < C < +\infty\)
such that for every \(m \in \mathcal P_p (\mathbb R^d)\)
and every \(x \in \mathbb R^d\),
\begin{equation}
\label{eq:boundedness-m-hat-density}
c e^{-U(x)} \leq \hat m(x) \leq C e^{-U(x)},
\end{equation}
Finally, there exists a constant \(L > 0\)
such that for every \(m\), \(m' \in \mathcal P_p (\mathbb R^d)\)
and every \(x \in \mathbb R^d\),
\begin{equation}
\label{eq:continuity-m-hat-density}
\lvert \hat m(x) - \hat m'(x) \rvert \leq L \mathcal W_p (m,m') e^{-U(x)}.
\end{equation}
\end{proposition}

\begin{proof}
Using \eqref{eq:boundedness-delta-F-delta-m}, we have
\begin{equation}
\label{eq:bound-m-hat-unnormalized-density}
\exp \biggl( -\frac{2}{\sigma^2} M_F - U (x) \biggr)
\leq \exp \biggl(- \frac{2}{\sigma^2} \frac{\delta F}{\delta m} (m, x) - U (x) \biggr)
\leq \exp \biggl( \frac{2}{\sigma^2} M_F - U (x) \biggr),
\end{equation}
and
\begin{multline}
\label{eq:bound-m-hat-normalization-constant}
\exp\biggl(-\frac{2M_F}{\sigma^2}\biggr) Z_0 \\
= \int_{\mathbb R^d} \exp \biggl(- \frac{2}{\sigma^2} M_F - U (x) \biggr) dx \leq Z_m \leq \int_{\mathbb R^d} \exp \biggl(\frac{2}{\sigma^2} M_F - U(x)\biggr)\,dx \\
= \exp\biggl(\frac{2M_F}{\sigma^2}\biggr) Z_0,
\end{multline}
Thus \(\hat m\) is well defined and
\eqref{eq:boundedness-m-hat-density} holds with constant
\(C = c^{-1} = \exp \bigl( 4M_F\sigma^{-2} \bigr)\).
Consequently,
\[
\int_{\mathbb R^d} \lvert x\rvert^p\,\hat m (dx)
\leq \int_{\mathbb R^d} \lvert x\rvert^p \hat m (x)\,dx
\leq \int_{\mathbb R^d} \lvert x\rvert^p C e^{-U(x)}\,dx < +\infty,
\]
that is, \(\hat m \in \mathcal P_p (\mathbb R^d)\).

Meanwhile, using the elementary inequality
\(\lvert e^x - e^y\rvert \leq e^{x\vee y} \lvert x-y\rvert\),
we have
\begin{multline*}
\biggl| \exp \biggl(- \frac{2}{\sigma^2} \frac{\delta F}{\delta m} (m, x) - U (x) \biggr)
- \exp \biggl(- \frac{2}{\sigma^2} \frac{\delta F}{\delta m} (m', x) - U (x) \biggr) \biggr| \\
\leq \frac{2}{\sigma^2} \exp \biggl(\frac{2M_F}{\sigma^2}\biggr)
\mathcal W_p (m,m') \exp \bigl( - U(x)\bigr).
\end{multline*}
Integrating the previous inequality with respect to \(x\), we obtain
\[
\lvert Z - Z'\rvert \leq \frac{2}{\sigma^2}
\exp \biggl(\frac{2M_F}{\sigma^2}\biggr) \mathcal W_p (m,m') Z_0.
\]
Using the bounds \eqref{eq:bound-m-hat-unnormalized-density} and \eqref{eq:bound-m-hat-normalization-constant},
we obtain the Lipschitz continuity \eqref{eq:continuity-m-hat-density}.
\end{proof}

The Lipschitz continuity \eqref{eq:continuity-m-hat-density}
implies the Hölder continuity of \(m \mapsto \hat m\).

\begin{corollary}
\label{corollary:wasserstein-lipschitz-phi}
Suppose Assumption~\ref{assum:wellposedness} holds.
Then the mapping \(\Phi : \mathcal P_p (\mathbb R^d) \to \mathcal P_p (\mathbb R^d)\) is \(1/p\)-H\"older continuous.
\end{corollary}

Before proving the corollary we show a lemma bounding the Wasserstein (coupling) distance between two probability measures by the \(L^\infty\) distance between their density functions.

\begin{lemma}
\label{lemma:bound-wasserstein-by-L-infinity}
Let \((X, d)\) be a metric space
and \(\mu\) be a Borel probability measure on \(X\).
Consider the space of positive integrable functions with respect to \(\mu\),
\[
L^\infty_{+,1} (\mu) \coloneqq \biggl\{f : X \to \mathbb R
~\text{Borel measurable}
: f \geq 0, \int f\,d\mu = 1\biggr\}.
\]
Equip \(L^\infty_{+,1} (\mu)\) with the usual \(L^\infty\) distance.
Suppose for some \(p \geq 1\) and some \(x_0 \in X\), we have
\(C_{\mu, p} \coloneqq \int_X d (x, x_0)^p\,\mu(dx) < +\infty\).
Then there exists a constant \(L_{\mu, p} > 0\)
such that for every \(f\), \(g \in L^\infty_{+,1} (\mu)\),
\[
\mathcal W_p (f\,\mu, g\,\mu )
\leq L_{\mu,p} \lVert f - g\rVert^{1/p}_{L^\infty},
\]
where \(f\,\mu\) is the probability measure
determined by \((f\,\mu)(A) \coloneqq \int_A f\,d\mu\) and similarily for \(g\).
\end{lemma}

\begin{proof}
Construct the following coupling \(\pi\) between \(f\,\mu\), \(g\,\mu\):
\begin{align*}
\pi &\coloneqq \pi_1 + \pi_2, \\
\pi_1 (dx\,dy) &\coloneqq (f\wedge g)(x)\,\mu_\Delta(dx\,dy), \\
\pi_2 (dx\,dy) &\coloneqq \biggl(\int (f - g)_+ (x) \mu(dx)\biggr)^{-1} (f - g)_+ (x) (g-f)_+ (y)\,\mu (dx)\,\mu (dy).
\end{align*}
Here \(\mu_\Delta\) is the measure supported on the diagonal \(\Delta \coloneqq \{(x, x) : x\in X\} \subset X\times X\) such that \(\mu_\Delta (A\times A) = \mu (A)\).
One readily verifies
that the projection mappings to the first and second variables,
denoted by \(X\), \(Y\) respectively, satisfy
\begin{align*}
X_\# \pi_1 = Y_\# \pi_1 &= (f \wedge g)\,\mu, \\
X_\# \pi_2 &= (f - g)_+\,\mu, \\
Y_\# \pi_2 &= (g - f)_+\,\mu
\end{align*}
Hence \(X_\# \pi = f\,\mu\), \(Y_\# \pi = g\,\mu\)
and \(\pi\) is indeed a coupling between \(f\,\mu\), \(g\,\mu\).

By the definition of Wasserstein distance, we obtain
\begin{multline*}
\mathcal W_p (f\,\mu, g\,\mu)^p
\leq
\int_{X\times X} d(x, y)^p \pi_1\,(dx\,dy)
+ \int_{X\times X} d(x, y)^p \pi_2\,(dx\,dy) \\
= \biggl(\int (f - g)_+\,\mu\biggr)^{-1}
\int_{X\times X} (f - g)_+(x) (g - f)_+(y) d(x, y)^p\,\mu (dx)\,\mu (dy).
\end{multline*}
Using triangle inequality
\(d(x,y)^p \leq C_p \bigl(d(x,x_0)^p + d(y,x_0)^p\bigr)\)
and exchanging \(x\), \(y\), the last term is again bounded by
\begin{align*}
&\hspace{-1em}\biggl(\int (f - g)_+\,\mu\biggr)^{-1}
\int_{X\times X} C_p \bigl(d(x,x_0)^p + d(y,x_0)^p\bigr)(f - g)_+ (x) (g-f)_+(y)\,\mu (dx)\,\mu (dy) \\
&= \frac {C_p}{\int (f - g)_+\,\mu} \int_{X\times X} d(x,x_0)^p
\biggl[(f-g)_+(x)(g-f)_+(y) \\
&\hphantom{
= \frac {C_p}{\int (f - g)_+\,\mu} \int_{X\times X} d(x,x_0)^p \biggl[}
\quad+ (g-f)_+(x)(f-g)_+(y)
\biggr]\,\mu (dx)\,\mu (dy) \\
&= C_p \int_X d(x,x_0)^p |f-g|(x)\,\mu(dx) \leq C_p C_{\mu,p} \Vert f - g \Vert_{L^\infty}.
\end{align*}
The Hölder constant is then given by \(L_{\mu,p} = (C_pC_{\mu,p})^{1/p}\).
\end{proof}

\begin{remark}
The Hölder exponent \(1/p\) in the inequality is sharp.
Consider the example: \(\mu = \mathrm{Leb}_{[0,1]}\),
\(f = (1+\varepsilon) \mathbf{1}_{\left[0,\frac 12\right)}
+ (1-\varepsilon)\mathbf{1}_{\left[\frac 12,1\right]}\),
\(g = (1-\varepsilon) \mathbf{1}_{[0,\frac 12)}
+ (1+\varepsilon) \mathbf{1}_{\left[\frac 12, 1\right]}\).
Then the \(\mathcal W_p\) distance between \(f\mu\), \(g\mu\)
is of order \(\varepsilon^{1/p}\) when \(\varepsilon\to 0\).
\end{remark}

\begin{proof}{\bf of Corollary~\ref{corollary:wasserstein-lipschitz-phi}\ }
Applying Lemma~\ref{lemma:bound-wasserstein-by-L-infinity} with \(\mu (dx) = e^{-U(x)}\,dx\),
we obtain
\[
\mathcal W_p (\hat m_1, \hat m_2) \leq L \biggl\Vert \frac{\hat m_1(x)}{e^{-U(x)}} - \frac{\hat m_2(x)}{e^{-U(x)}} \biggr\Vert_{L^\infty}^{1/p},
\]
while by \eqref{eq:continuity-m-hat-density} we have
\[
\biggl\Vert \frac{\hat m_1(x)}{e^{-U(x)}}
- \frac{\hat m_2(x)}{e^{-U(x)}} \biggr\Vert_{L^\infty}
\leq L \mathcal W_p (m_1, m_2).
\]
The Hölder continuity follows.
\end{proof}

\begin{proof}{\bf of Proposition~\ref{prop:wellposedness-time-dynamics}\ }
\emph{Existence.} We will use Schauder's fixed point theorem.
To this end, fix \(T > 0\), let \(m_0 \in \mathcal P_p\) be the initial value
and denote \(X = C([0,T]; \mathcal W_p)\).
Let \(\mathbf T : X \to X\) be the mapping determined by
\begin{equation}
\label{eq:duhamel}
\mathbf {T} [m]_t \coloneqq \int_0^t \alpha e^{-\alpha (t - s)}\,\hat m_s\,ds
+ e^{-\alpha t}\,m_0
= \int_0^t \alpha e^{-\alpha (t - s)}\,\Phi (m_s)\,ds + e^{-\alpha t}\,m_0,
\quad t \in [0,T].
\end{equation}
We verify indeed \(\mathbf{T}[m] \in X\),
i.e. \(\mathbf{T}[m]_t \in \mathcal P_p\) for every \(t \in [0,T]\),
and \(t \mapsto \mathbf T[m]_t\) is continuous with respect to \(\mathcal W_p\).
This first claim follows from the fact
that \(\mathbf T[m]_t\) is a convex combination of elements in \(\mathcal P_p\),
as we have shown \(\hat m_s = \Phi(m_s) \in \mathcal P_p (\mathbb R^d)\).
The second claim follows from
\begin{multline}
\label{eq:time-regularity-Tm}
\mathcal W_p(\mathbf T[m]_{t+\delta}, \mathbf T[m]_t)^p
\leq \alpha \int_0^\delta e^{-\alpha (\delta - s )}
\mathcal W_p(\hat m_s, m_t)^p\,ds \\
\leq C (1 - e^{-\alpha \delta})
\bigl({\textstyle\sup}_{\hat m \in \Image \Phi} M_p(\hat m)
+ M_p (\mathbf T[m]_t)\bigr).
\end{multline}
Next we show the compactness of the mapping \(\mathbf T\).
Setting \(t = 0\) in the previous equation
and letting \(\delta\) vary in \([0,T]\),
we obtain
\[
\sup_{m\in X} \sup_{t \in [0,T]} M_p \bigl(\mathbf{T}[m]_t\bigr) \leq C.
\]
Plugging this back to \eqref{eq:time-regularity-Tm},
we have
\begin{equation}
\label{eq:wasserstein-time-regularity}
\sup_{m\in X, 0\leq t<t+\delta\leq T}
\mathcal W_p\bigl(\mathbf{T}[m]_{t+\delta}, \mathbf{T}[m]_t\bigr)
\leq C \delta^{1/p}.
\end{equation}
From \eqref{eq:boundedness-delta-F-delta-m}
one knows that \(\Image\Phi\) forms a precompact set in \(\mathcal P_p\),
and since \(X_t \coloneqq \{\mathbf T[m]_t : m\in X\}\) lies in the convex combination of \(\Image\Phi\) and \(\{m_0\}\), \(X_t\) is also precompact.
Then by the Arzel\`a--Ascoli theorem,
\(\Image \mathbf{T} = \mathbf T[X]\) is a precompact set.
In other words, \(\mathbf T\) is a compact mapping.
We use Schauder's theorem to conclude that \(\mathbf T\) admits a fixed point,
i.e. \eqref{eq:efp} admits at least one solution in \(X\).

\emph{Wellposedness when \(p=1\).}
The mapping \(\Phi\) is Lipschitz in this case.
The wellposedness follows from standard Picard--Lipschitz arguments.

\emph{Pointwise solution.}
By definition, \(\hat m_t\) admits the density function
\[
\hat m_t (x) = \frac{1}{Z_t} \exp \biggl( - \frac{2}{\sigma^2} \frac{\delta F}{\delta m} (m_t, x) - U(x) \biggr),
\]
where
\(Z_t \coloneqq \int_{\mathbb R^d}
\exp \bigl( - \frac{2}{\sigma^2} \frac{\delta F}{\delta m} (m_t, x)
- U(x) \bigr)\,dx\)
is the normalization constant.
The functional derivative \(\frac{\delta F}{\delta m} (m_t, x)\) is continuous in \(t\)
by the continuities of \(t\mapsto m_t\)
and \(m\mapsto \frac{\delta F}{\delta m} (m, x)\),
and is bounded for every \(t \geq 0\).
By the dominated convergence theorem, both
\( \exp \bigl( - \frac{2}{\sigma^2} \frac{\delta F}{\delta m} (m_t, x)
- U(x) \bigr) \)
and \(Z_t\) are continuous in \(t\)
and bounded.
Hence \(t \mapsto \hat m_t (x)\) is continuous
and bounded uniformly in \(x\).
Suppose now the initial value \(m_0\) has density \(m_0(x)\).
Define the density of \(m_t\) according to the Duhamel's formula \eqref{eq:duhamel}:
\begin{equation}
\label{eq:duhamel-density}
m_t (x) \coloneqq \int_0^t \alpha e^{-\alpha (t - s)} \hat m_s(x)\,ds
+ e^{-\alpha t} m_0(x), \quad \text{for}~x \in \mathbb R^d.
\end{equation}
By definition \(m_t (x)\) defined by \eqref{eq:duhamel-density}
is indeed the density of \(m_t\) solving the time dynamics \eqref{eq:efp},
and is automatically continuous in \(t\).
Since \(\alpha e^{-\alpha (t - s)} \hat m_s(x)\) in \eqref{eq:duhamel-density} is continuous and bounded in \(s\) for every \(t \geq 0\),
the density \(m_t(x)\) is \(C^1\) in \(t\)
and satisfies the pointwise equality \eqref{eq:efp-pointwise}.
\end{proof}

We also obtain a density bound that will be used in the following.

\begin{corollary}
Suppose Assumption~\ref{assum:wellposedness} holds.
There exist constants \(c\), \(C > 0\), depending only on \(F\) and \(U\),
such that
\begin{align}
\label{eq:m-density-lower-bound}
m_t (x) &\geq ( 1 - e^{-\alpha t} ) c e^{-U(x)}, \\
\label{eq:m-density-upper-bound}
m_t (x) &\leq ( 1 - e^{-\alpha t} ) C e^{-U(x)} + e^{-\alpha t} m_0(x),
\end{align}
for every \(x \in \mathbb R^d\).
\end{corollary}

\begin{proof}
For all \(\hat m \in \Image \Phi\), we have
\[
\hat m(x) \geq c e^{-U(x)}.
\]
Then by the definition of density \eqref{eq:duhamel-density}, we have
\begin{align*}
m_t (x) &\geq \int_0^t \alpha e^{-\alpha (t - s)} \hat m_s(x)\,ds \\
&\geq c e^{-U(x)} \int_0^t \alpha e^{-\alpha (t - s)} \hat m_s(x)\,ds \\
&= ( 1 - e^{-\alpha t} ) c e^{-U(x)}.
\end{align*}
The proof for the upper bound is similar.
\end{proof}

\subsection{Proof of Theorem~\ref{thm:general-convergence}}

As it is important to our proof of Theorem~\ref{thm:general-convergence},
we single out the derivative in time result in the following proposition
and prove it before tackling the other parts of the theorem.

\begin{proposition}
\label{prop:derivative-value-function}
Suppose Assumptions~\ref{assum:wellposedness}
and \ref{assum:convergence} holds,
and let \((m_t)_{t \geq 0}\) be a solution
to \eqref{eq:efp} in \(\mathcal W_p\).
Then for every \(t > 0\),
\begin{equation}
\label{eq:derivative-value-function}
\frac{dV^\sigma (m_t)}{dt} = - \frac{\alpha \sigma^2}{2}
\bigl( H(m_t | \hat m_t) + H(\hat m_t | m_t)\bigr).
\end{equation}
\end{proposition}

Before proving the proposition, we show a lemma on the uniform integrability
of \(m_t\) and \(\hat m_t\).

\begin{lemma}
\label{lemma:integrability-derivative-of-entropy}
Fix \(s > 0\).
Under the conditions of the previous proposition,
there exist integrable functions \(f\), \(g\)
such that for every \(t \in [s, +\infty)\) and every \(x \in \mathbb R^d\),
\[
g(x) \leq \log \frac{m_t(x)}{e^{-U(x)}} \bigl(\hat m_t(x) - m_t(x)\bigr)
\leq f(x).
\]
\end{lemma}

\begin{proof}
We first deal with the first term \(\log \frac{m_t(x)}{e^{-U(x)}} \hat m_t(x)\).
Using the bounds~\eqref{eq:m-density-lower-bound},
\eqref{eq:m-density-upper-bound} we have
\begin{align*}
\log \frac{m_t(x)}{e^{-U(x)}} \hat m_t(x)
&\geq \log \frac{ (1 - e^{-\alpha t} ) c e^{-U(x)} }{ e^{-U(x)} } \hat m_t (x)
= \log \bigl( (1 - e^{-\alpha t} ) c \bigr) \hat m_t (x) \\
&\geq \log \bigl( (1 - e^{-\alpha s} \bigr) c ) \hat m_t (x)
\geq \log \bigl( (1 - e^{-\alpha s} ) c \bigr) C e^{-U(x)} \eqqcolon g_1(x).
\end{align*}
Here we shrink the constant \(c\) if necessary so that \(c < 1\)
and in the last inequality the coefficient
\(\log \bigl( (1 - e^{-\alpha s} ) c \bigr)\) is negative.
Now we upper bound \(\log \frac{m_t(x)}{e^{-U(x)}} \hat m_t(x)\).
We have
\begin{align*}
\log \frac{m_t(x)}{e^{-U(x)}}
&\leq \log \biggl( e^{-\alpha t} \frac{m_0 (x)}{e^{-U(x)}}
+ \int_0^t \alpha e^{-\alpha (t-s)} \frac{\hat m_s (x)}{e^{-U(x)}}\,ds \biggr) \\
&\leq \log \biggl( e^{-\alpha t} \frac{m_0 (x)}{e^{-U(x)}}
+ C\int_0^t \alpha e^{-\alpha (t-s)}\,ds\biggr)
= \log \biggl( e^{-\alpha t} \frac{m_0 (x)}{e^{-U(x)}} + C(1 - e^{-\alpha t}) \biggr) \\
&\leq \log \bigl( (1 - e^{-\alpha t}) C \bigr)
+ \frac{e^{-\alpha t}}{C (1 - e^{-\alpha t} )} \frac{m_0 (x)}{e^{-U(x)}} \leq \log C + C_{s} \frac{m_0 (x)}{e^{-U(x)}}.
\end{align*}
Here in the third inequality we used the elementary inequality
\(\log (x + y) \leq \log x + \frac{y}{x}\) for real \(x\), \(y\),
and in the last line we maximize over \(t \geq s\)
and set
\(C_{s} = e^{-\alpha s} \bigl(C (1 - e^{-\alpha s} )\bigr)^{-1}\).
Therefore,
\begin{align*}
\log \frac{m_t(x)}{e^{-U(x)}} \hat m_t (x)
&\leq \biggl(\log C + C_{s} \frac{m_0 (x)}{e^{-U(x)}} \biggr) \hat m_t (x)
\leq \biggl(\log C + C_{s} \frac{m_0 (x)}{e^{-U(x)}}\biggr) C e^{-U(x)} \\
&= \log C \cdot C e^{-U(x)} + C_{s}C m_0(x) \eqqcolon f_1(x).
\end{align*}

Now consider the second term \(\log \frac{m_t(x)}{e^{-U(x)}} m_t(x)\).
Applying Jensen's inequality
to the Duhamel formula \eqref{eq:duhamel-density},
we have
\begin{align*}
\log \frac{m_t(x)}{e^{-U(x)}} m_t(x)
&\leq e^{-\alpha t} \log \frac{m_0(x)}{e^{-U(x)}} m_0(x)
+ \int_0^t \alpha e^{-\alpha (t-s)} \log \frac{\hat m_s(x)}{e^{-U(x)}} \hat m_s(x)\,dt \\
&\leq e^{-\alpha t} \log \frac{m_0(x)}{e^{-U(x)}} m_0(x)
+ \int_0^t \alpha e^{-\alpha (t-s)} \log C \cdot \hat m_s (x)\,dt \\
&\leq e^{-\alpha t} \log \frac{m_0(x)}{e^{-U(x)}} m_0(x)
+ \int_0^t \alpha e^{-\alpha (t-s)} \log C \cdot C e^{-U(x)}\,dt \\
&\leq \biggl(\log \frac{m_0(x)}{e^{-U(x)}} m_0(x) \biggr)_+
+ \log C \cdot C e^{-U(x)} \eqqcolon -g_2 (x)
\end{align*}
In the second and third inequality we use consecutively the bound \(\hat m (x) \leq C e^{-U(x)}\) with \(C > 1\).
For the lower bound of the second term we note
\[
\log \frac{m_t(x)}{e^{-U(x)}} m_t(x)
= \log \frac{m_t(x)}{e^{-U(x)}} \cdot \frac{m_t(x)}{e^{-U(x)}} e^{-U(x)} \geq -\frac 1e e^{-U(x)} \eqqcolon -f_2 (x)
\]
The proof is complete by letting
\(f = f_1 + f_2\) and \(g = g_1 + g_2\).
\end{proof}

\begin{proof}{\bf of Proposition~\ref{prop:derivative-value-function}\ }
Thanks to the lemma above,
we can apply the dominated convergence theorem
to differentiate \(t \mapsto V^\sigma(m_t)\) and obtain
\[
\frac{d H (m_t)}{dt} = \alpha \int_{\mathbb R^d}
\bigl(\log m_t (x) + U(x)\bigr) \bigl( \hat m_t (x) - m_t(x) \bigr)\,dx.
\]
For the regular term \(F(m_t)\), by the definition of functional derivative,
we have
\[
F(m_{t+\delta}) - F(m_t)
= \int_{0}^{1} \int_{\mathbb R^d} \frac{\delta F}{\delta m} (m_{t + u\delta}, x)
\bigl( m_{t+\delta} (x) - m_t (x) \bigr)\,dx\,du.
\]
Applying again the dominated convergence theorem,
the derivative reads
\begin{align*}
\frac{dV^\sigma (m_t)}{dt}
&= \alpha \int_{\mathbb R^d} \biggl( \frac{\delta F}{\delta m} (m_t, x) + \frac{\sigma^2}{2} \log m_t(x) + \frac{\sigma^2}{2} U(x) \biggr)
\bigl(\hat m_t (x) - m_t(x) \bigr)\,dx \\
&= \alpha \int_{\mathbb R^d} \biggl(C_t + \frac{\sigma^2}{2}
\log m_t(x) - \frac{\sigma^2}{2} \log \hat m_t(x) \biggr)
\bigl(\hat m_t (x) - m_t(x) \bigr)\,dx \\
&= \alpha \int_{\mathbb R^d} \biggl(\frac{\sigma^2}{2}
\log m_t(x) - \frac{\sigma^2}{2} \log \hat m_t(x) \biggr)
\bigl(\hat m_t (x) - m_t(x) \bigr)\,dx \\
&= - \frac{\alpha\sigma^2}{2} \bigl( H(m_t | \hat m_t) + H(\hat m_t | m_t)\bigr),
\end{align*}
where in the second line we use the first-order condition for \(\hat m_t\) and \(C_t\) is a constant that may depend on \(t\).
\end{proof}

\begin{remark}
The result of Proposition~\ref{prop:derivative-value-function} implies
\begin{itemize}
\item \(\int_0^{+\infty} \bigl( H(m_t | \hat m_t) + H(\hat m_t | m_t) \bigr)\,dt
< +\infty\);

\item The derivative \(\frac{dV^\sigma (m_t)}{dt}\) vanishes
if and only if \(m_t = \hat m_t\), i.e. the dynamics reaches a stationary point.
\end{itemize}
\end{remark}

\begin{proof}{\bf of Theorem~\ref{thm:general-convergence}\ }
Our strategy of proof is as follows.
First we show that, by the (pre-)compactness of the flow
\((m_t)_{t \geq 0} \) in a suitable Wasserstein space,
the flow converges up to an extraction of subsequence.
Then we prove by a monotonicity argument the convergence holds true
without extraction.
Finally we study the convergence of the density functions
and prove the convergence of value function
by the dominated convergence theorem.

According to the Duhamel's formula \eqref{eq:duhamel},
the measure \(m_t\) is a (weighted) linear combination
of the initial value \(m_0\) and the best responses \(\hat m_s\).
Since there exists some \(p' > p\)
such that \(m_0 \in \mathcal P_{p'} (\mathbb R^d)\),
we obtain by the triangle inequality
\begin{multline*}
\Vert m_t\Vert_{p'}^{p'}
\leq e^{-\alpha t} \Vert m_0 \Vert_{p'}^{p'}
+ (1 - e^{-\alpha t} ) \sup_{0 \leq s\leq t} \Vert \hat m_s \Vert_{p'}^{p'} \\
\leq \Vert m_0 \Vert_{p'}^{p'}
+ \sup_{m \in \mathcal P_p (\mathbb R^d)} \Vert \Phi (m) \Vert_{p'}^{p'}
\leq \Vert m_0 \Vert_{p'}^{p'} + C \int_{\mathbb R^d} x^{p'} e^{-U(x)}\,dx.
\end{multline*}
Thus the flow \((m_t)_{t \geq 0} \) in precompact in \(\mathcal P_p (\mathbb R^d)\)
and the set of limit points,
\[
w(m_0) \coloneqq \{ m \in \mathcal P_p (\mathbb R^d) : \exists t_{n} \to +\infty~\text{such that}~m_{t_n} \to m\},
\]
is nonempty.
We now show that \(w(m_0)\) is the singleton \(\{ m^* \}\)
and therefore \(m_t \to m^*\) in \(\mathcal W_p\).
Pick \(m \in w(m_0)\) and let \((t_n)_{n\in\mathbb N}\) be an increasing sequence
such that \(t_n \to +\infty\) and \(m_{t_n} \to m\).
Extracting a subsequence if necessary,
we may suppose \(t_{n+1} - t_n \geq 1\) for \(n \in \mathbb N\).
Proposition~\ref{prop:derivative-value-function} implies
for every \(t\), \(s\) such that \(t > s \geq 0\),
\[
V^\sigma (m_s) - V^\sigma (m_t)
= \int_s^t \bigl( H(m_u | \hat m_u) + H(\hat m_u | m_u) \bigr)\,du.
\]
Consequently,
\begin{align*}
V^\sigma(m_0)
&\geq V^\sigma(m_{t_0}) - V^\sigma(m_{t_n}) \\
&\geq \sum_{k=0}^{n-1} \int_{t_{k}}^{t_{k} + 1}
\bigl( H(m_u | \hat m_u) + H(\hat m_u | m_u) \bigr)\,du \\
&\geq \sum_{k=0}^{n-1} \int_0^1
\bigl( H(m_{t_k+u} | \hat m_{t_k+u}) + H(\hat m_{t_k+u} | m_{t_k+u}) \bigr)\,du.
\end{align*}
By taking \(n\to +\infty\), we obtain
\[
\sum_{k=0}^{n-1} \int_0^1
\bigl( H(m_{t_k+u} | \hat m_{t_k+u}) + H(\hat m_{t_k+u} | m_{t_k+u}) \bigr)\,du
< +\infty.
\]
Therefore,
\begin{align*}
0 &= \lim_{k \to +\infty} \int_0^1
\bigl( H(m_{t_k+u} | \hat m_{t_k+u}) + H(\hat m_{t_k+u} | m_{t_k+u}) \bigr)\,du \\
&\geq \int_0^1 \liminf_{k \to +\infty} \bigl( H(m_{t_k+u} | \hat m_{t_k+u}) + H(\hat m_{t_k+u} | m_{t_k+u}) \bigr)\,du \\
&= \int_0^1 \liminf_{k \to +\infty }
\Bigl( H\bigl(S[u]m_{t_k}\big|\Phi(S[u]m_{t_k})\bigr)
+ H\bigl(\Phi(S[u]m_{t_k})\big|S[u]m_{t_k}\bigr)
\Bigr)\,du \\
&= \int_0^1
\Bigl( H\bigl(S[u]m\big|\Phi(S[u]m)\bigr)
+ H\bigl(\Phi(S[u]m)\big|S[u]m\bigr)
\Bigr)\,du.
\end{align*}
In the first inequality we applied Fatou's lemma,
and in the last equality we used the convergence \(m_{t_k} \to m\),
the continuity of \(S[u]\) and \(\Phi\),
and the joint lower-semicontinuity of \((\mu, \nu) \mapsto H(\mu | \nu)\)
with respect to the weak convergence of measures.
Then we have
\[
H\bigl(S[u]m\big|\Phi(S[u]m)\bigr)
+ H\bigl(\Phi(S[u]m)\big|S[u]m\bigr) = 0
\]
for a.e. \(u \in [0,1]\).
Using again the lower-semicontinuity of relative entropy, we obtain
\[
H \bigl( m \big| \Phi(m) \bigr) + H \bigl( \Phi(m) \big| m \bigr)
\leq \liminf_{u \to 0}
\Bigl( H\bigl(S[u]m\big|\Phi(S[u]m)\bigr)
+ H\bigl(\Phi(S[u]m)\big|S[u]m\bigr)\Bigr) = 0.
\]
That is to say, as a probability measure \(m = \Phi(m) = \hat m\).
By our assumption \(\Phi\) has unique fixed point \(m^*\),
therefore \(m = m^*\) and \(w(m_0)\) is equal to the singleton
\(\{m^*\}\).

Next we show that the convergence of the density function
\(m_t(\cdot) \to m^*(\cdot)\).
Since \(\frac{\sigma^2}{2} H(m^*) \leq V^\sigma (m^*) < +\infty\),
the measure \(m^*\) has a density function, which we denote by \(m^*(\cdot)\).
The Duhamel's formula for density functions \eqref{eq:duhamel-density} yields
\begin{align*}
\bigl|m_t (x) - m^*(x)\bigr|
&\leq e^{-\alpha t} \bigl| m_0(x) - m^*(x) \bigr|
+ \int_0^t \alpha e^{-\alpha (t-s)} \bigl| \hat m_s(x) - m^*(x) \bigr|\,ds \\
&\leq e^{-\alpha t} \bigl| m_0(x) - m^*(x) \bigr|
+ \int_0^t \alpha e^{-\alpha (t-s)} L \mathcal W_p (\hat m_s, m^*) e^{-U(x)}\,ds \\
&= e^{-\alpha t} \bigl| m_0(x) - m^*(x) \bigr|
+ \int_0^{t} \alpha e^{-\alpha s} L \mathcal W_p (\hat m_{t-s}, m^*) e^{-U(x)}\,ds \\
&= e^{-\alpha t} \bigl| m_0(x) - m^*(x) \bigr|
+ \int_0^{+\infty} \mathbf{1}_{s \leq t} \alpha e^{-\alpha s} L \mathcal W_p (\hat m_{t-s}, m^*) e^{-U(x)}\,ds.
\end{align*}
The integrand in the last integral is positive and upper-bounded by the integrable function
\[
\mathbf{1}_{s \leq t} \alpha e^{-\alpha s} L \mathcal W_p (\hat m_{t-s}, m^*) e^{-U(x)}
\leq \alpha L \sup_{t\geq 0}\mathcal W_p (\hat m_{t}, m^*) e^{-\alpha s} e^{-U(x)},
\]
where \(\sup_{t\geq 0}\mathcal W_p (\hat m_{t}, m^*) < +\infty\)
because \((m_t)_{t \geq 0}\)
is a continuous and convergent flow in \(\mathcal P_p\).
Hence by the dominated convergence theorem,
\begin{multline*}
\lim_{t \to +\infty} \int_0^{+\infty} \mathbf{1}_{s \leq t} \alpha e^{-\alpha s} L \mathcal W_p (\hat m_{t-s}, m^*) e^{-U(x)}\,ds
\\
= \int_0^{+\infty} \lim_{t \to +\infty} \mathbf{1}_{s \leq t} \alpha e^{-\alpha s} L \mathcal W_p (\hat m_{t-s}, m^*) e^{-U(x)}\,ds = 0,
\end{multline*}
where
\(\lim_{s \to +\infty}\mathcal{W}_p (\hat m_s, m^*)
= \lim_{s \to +\infty}\mathcal{W}_p \bigl( \Phi(m_s), m^*\bigr)
= 0\)
since \(m_s \to m^*\) and \(\Phi\) is continuous.
As a result, \(m_t (x) \to m^*(x)\) when \(t \to +\infty\).
We finally show the convergence of the value function.
Note that, as in the proof of Proposition~\ref{prop:derivative-value-function},
the entropic term is doubly bounded by integrable functions
\[
-f_2 (x) \leq m_t(x) \log \frac{m_t(x)}{e^{-U (x) }} \leq -g_2(x).
\]
Applying the dominated convergence theorem,
we obtain
\begin{multline*}
\lim_{t\to +\infty} H(m_t)
= \lim_{t\to+\infty} \int_{\mathbb R^d} m_t(x) \log \frac{m_t(x)}{e^{-U (x)}}\,dx
= \int_{\mathbb R^d} \lim_{t\to+\infty} m_t(x) \log \frac{m_t(x)}{e^{-U (x)}}\,dx \\
= \int_{\mathbb R^d} m^*(x) \log \frac{m^*(x)}{e^{-U (x) }}\,dx = H(m^*).
\end{multline*}
The convergence in Wasserstein distance implies already
\(F(m_t) \to F(m^*)\).
Therefore \(\lim_{t\to +\infty} V^\sigma (m_t) = V^\sigma(m^*)\).
\end{proof}

\subsection{Proof of Theorem
\ref{thm:exp}}

We again show some technical results before moving on to the proof
of the theorem.

\begin{lemma}
\label{lemma:partial-derivative-m-hat-zero}
Suppose Assumptions~\ref{assum:wellposedness} and \ref{assum:convergence} holds,
and let \(m_t\) be a solution to \eqref{eq:efp}.
For every \(t > 0\), we have
\begin{multline*}
0 \leq \int_{\mathbb R^d} \hat m_{t+\delta} (x)
\biggl( \log \hat m_{t+\delta} (x) + U(x)
+ \frac{2}{\sigma^2} \frac{\delta F}{\delta m } (m_t, x) \biggr)\,dx \\
- \int_{\mathbb R^d} \hat m_{t} (x) \biggl(\log \hat m_{t} (x) + U(x)
+ \frac{2}{\sigma^2} \frac{\delta F}{\delta m } (m_t, x) \biggr)\,dx
= O(\delta^{1/p})
\end{multline*}
when \(\delta \to 0\).
\end{lemma}

\begin{proof}
Denote the quantity to bound by \(I\).
We write it as the sum of the following two terms:
\begin{align*}
I &= I_1 + I_2, \\
I_1 &= \int_{\mathbb R^d} \biggl( \frac{2}{\sigma^2} \frac{\delta F}{\delta m } (m_t, x)
+ \log \hat m_t (x) + U (x) \biggr)
\bigl( \hat m_{t+\delta} (x) - \hat m_{t} (x) \bigr)\,dx
= 0, \\
I_2 &= \int_{\mathbb R^d}
\bigl( \log \hat m_{t+\delta} (x) - \log \hat m_{t} (x) \bigr)
\hat m_{t+\delta} (x)\,dx.
\end{align*}
The term \(I_1\) is zero
because \(\frac{2}{\sigma^2} \frac{\delta F}{\delta m } (m_t, x) + \log \hat m_t (x)\)
is constant by the first-order condition.
On the other hand, we have
\(I_2 = H(\hat m_{t+\delta} | \hat m_t) \geq 0 \).
Let us bound the other side.
Since \(\hat m_s (x) \geq c e^{-U(x)} \) holds for every \(s \geq 0\),
we have
\begin{align*}
\bigl| \log \hat m_{t+\delta} (x) - \log \hat m_t (x) \bigr| \hat m_{t+\delta} (x)
&\leq \frac{\hat m(x)}
{\min \bigl\{ \hat m_{t+\delta}(x), \hat m_t(x) \bigr\} }
\bigl(m_{t+\delta}(x) - \hat m_t(x)\bigr) \\
&\leq C \bigl(\hat m_{t+\delta}(x) - \hat m_t(x)\bigr) \\\
&\leq C e^{-U(x)} \mathcal W_p(m_{t+\delta}, m_t) \\
&\leq C e^{-U(x)} \delta^{1/p}.
\end{align*}
Here we have used \(\log \frac xy \leq \frac {|x-y|}{\min\{x,y\}}\) in the first inequality,
\eqref{eq:continuity-m-hat-density} in the second inequality,
and \eqref{eq:wasserstein-time-regularity} in the last inequality.
\end{proof}

We need the following notion to treat the possibly non-differentiable
relative entropy.

\begin{definition}
For a real function \(f : (t-\varepsilon, t+\varepsilon) \to \mathbb R\)
defined on a neighborhood of \(t\),
the set of its upper-differentials at \(t\) is
\[
D^+ f (t) \coloneqq \biggl\{ p \in \mathbb R :
\limsup_{s \to t} \frac{f(s) - f(t) - p (s - t)} {|s - t|} \leq 0 \biggr\}.
\]
Lower-differentials are defined as \(D^- f (t) \coloneqq - D^+ (-f) (t)\).
\end{definition}

\begin{lemma}
\label{lemma:lower-differential-integration}
Let \(f : [a, b] \to \mathbb R\) be a function defined on a closed interval, continuous on its two ends \(a\) and \(b\).
If \(f\) has nonnegative lower-differentials on \((a, b)\),
i.e. for every \(a < t < b\)
there exists \(p_t \in D^- f (t)\) with \(p_t \geq 0\),
then \(f(b) \geq f(a)\).
\end{lemma}

\begin{proof}
Since the interval \([a,b]\) is compact,
for every \(\varepsilon > 0\),
we can find a finite sequence \(a < x_1 < \ldots < x_n < b\)
such that \(f(x_{i+1}) - f(x_i) \geq - \varepsilon (x_{i+1} - x_i )\)
with \(x_1 < a + \varepsilon\) and \(b < x_n + \varepsilon\).
Thus we have \(f(x_n) - f(x_1) \geq - \varepsilon (x_n - x_1) \).
We conclude by taking the limit \(\varepsilon \to 0\).
\end{proof}

Next we calculate the upper-differential of the relative entropy
\(t \mapsto H(m_t | \hat m_t)\).

\begin{proposition}
\label{prop:entropy-exponential}
Let Assumptions~\ref{assum:wellposedness}, \ref{assum:convergence}
and \ref{assum:convex} hold,
and let \((m_t)_{t \geq 0}\) be a solution to \eqref{eq:efp}
in \(\mathcal W_p\).
Then the relative entropy \(H : t \mapsto H(m_t | \hat m_t)\) is continuous
on \([0, +\infty)\),
and for every \(t > 0\),
the set of upper differentials \(D^+ H(t)\) is non-empty
and there exists \(p_t \in D^+ H (t)\) such that
\[
p_t \leq - \alpha \bigl( H(m_t | \hat m_t) + H(\hat m_t | m_t) \bigr).
\]
\end{proposition}

\begin{proof}
Fix \(t > 0\). The relative entropy reads
\begin{align*}
H_t \coloneqq H(m_t|\hat m_t)
&= \int_{\mathbb R^d} m_t(x) \bigl(\log m_t (x) - \log \hat m_t (x)\bigr)\,dx \\
&= \int_{\mathbb R^d} m_t(x) \biggl(\log m_t (x) + U(x) + \frac{2}{\sigma^2} \frac{\delta F}{\delta m} (m_t,x) \biggr)\,dx \\
&\quad - \int_{\mathbb R^d} m_t(x) \biggl(\log \hat m_t (x) + U(x) + \frac{2}{\sigma^2} \frac{\delta F}{\delta m} (m_t,x) \biggr)\,dx \\
&= \int_{\mathbb R^d} m_t(x) \biggl(\log m_t (x) + U(x) + \frac{2}{\sigma^2} \frac{\delta F}{\delta m} (m_t,x) \biggr)\,dx \\
&\quad - \int_{\mathbb R^d} \hat m_t(x) \biggl(\log \hat m_t (x) + U(x) + \frac{2}{\sigma^2} \frac{\delta F}{\delta m} (m_t,x) \biggr)\,dx \\
&\eqqcolon H_{1,t} - H_{2,t}.
\end{align*}
In the second equality we can separate the integral into two parts
because the integrand of the second term
\(m_t(x) \bigl(\log \hat m_t (x) + U(x) + \frac{2}{\sigma^2} \frac{\delta F}{\delta m} (m,x) \bigr)\)
is integrable as it is constant by the first-order condition.
For the same reason, in the fourth equality
we can replace \(m_t\) by \(\hat m_t\) in the second term,
as we are integrating against a constant
and \(m_t, \hat m_t\) have the same total mass \(1\).

Now we consider the difference \(H_{t+\delta} - H_t = (H_{1,t+\delta} - H_{1,t}) - (H_{2, t+\delta} - H_{2,t})\).
For the first part we have
\begin{align*}
H_{1,t+\delta} - H_{1,t} &= H(m_{t+\delta}) - H(m_t) \\
&\quad+ \frac{2}{\sigma^2}\int_{\mathbb R^d}
\biggl( m_{t+\delta}(x) \frac{\delta F}{\delta m} (m_{t+\delta},x) - m_t(x) \frac{\delta F}{\delta m} (m_t,x) \biggr)\,dx. \\
&= \delta \int_{\mathbb R^d} \alpha \log \frac{m_t(x)}{e^{-U(x)}} (\hat m_t (x) - m_t(x))\,dx \\
&\quad+ \frac{2 \delta}{\sigma^2}\int_{\mathbb R^d}
\alpha \bigl(\hat m_t (x) - m_t(x)\bigr) \frac{\delta F}{\delta m} (m_{t},x)\,dx \\
&\quad+ \frac{2\delta}{\sigma^2}\int_{\mathbb R^d}\int_{\mathbb R^d} m_t(x) \frac{\delta^2 F}{\delta m^2} (m_t, x, y) \alpha \bigl(\hat m_t (y) - m_t(y)\bigr)\,dx\,dy + o(\delta) \\
&= \alpha \delta \int_{\mathbb R^d} \biggl(\log m_t(x) + U(x) + \frac{2}{\sigma^2} \frac{\delta F}{\delta m} (m_t, x) \biggr)
\bigl(\hat m_t(x) - m_t(x)\bigr)\,dx \\
&\quad+ \frac{2\delta}{\sigma^2}\int_{\mathbb R^d}\int_{\mathbb R^d} m_t(x) \frac{\delta^2 F}{\delta m^2} (m_t, x, y) \alpha
\bigl(\hat m_t (y) - m_t(y)\bigr)\,dx\,dy + o(\delta).
\end{align*}
by Lemma~\ref{lemma:integrability-derivative-of-entropy} and dominated convergence theorem.

Next we calculate the second part:
\begin{align*}
H_{2,t+\delta} - H_{2,t} &= \frac{2}{\sigma^2} \int_{\mathbb R^d} \hat m_{t+\delta} (x)
\biggl(\frac{\delta F}{\delta m} (m_{t+\delta},x) - \frac{\delta F}{\delta m} (m_{t},x)\biggr)\,dx \\
&\quad + \biggl[ \int_{\mathbb R^d} \hat m_{t+\delta} (x)
\biggl(\log \hat m_{t+\delta} (x) + U(x) + \frac{2}{\sigma^2} \frac{\delta F}{\delta m } (m_t, x)\biggr)\,dx \\
&\hspace{3em} - \int_{\mathbb R^d} \hat m_{t} (x)
\biggl(\log \hat m_{t} (x) + U(x) + \frac{2}{\sigma^2} \frac{\delta F}{\delta m } (m_t, x) \biggr)\,dx \biggl]
\end{align*}
For the first difference we use the expansion
\(m_{t + \delta} - m_t = \alpha \delta\,(\hat m_t - m_t ) + o(\delta)\)
and apply the dominated convergence theorem to obtain
\begin{align*}
&\frac{2}{\sigma^2} \int_{\mathbb R^d} \hat m_{t+\delta} (x)
\biggl(\frac{\delta F}{\delta m} (m_{t+\delta},x) - \frac{\delta F}{\delta m} (m_{t},x)\biggr)\,dx \\
&= \frac{2 \alpha \delta}{\sigma^2} \int_{\mathbb R^d} \int_{\mathbb R^d} \hat m_{t+\delta} (x) \frac{\delta^2 F}{\delta m^2} (m_t, x, y) \alpha
\bigl(\hat m_t (y) - m_t(y) \bigr)\,dx\,dy + o(\delta) \\
&= \frac{2 \alpha \delta}{\sigma^2} \int_{\mathbb R^d} \int_{\mathbb R^d} \hat m_{t} (x) \frac{\delta^2 F}{\delta m^2} (m_t, x, y) \alpha
\bigl(\hat m_t (y) - m_t(y)\bigr)\,dx\,dy + o(\delta)
\end{align*}
and the second difference is already treated in Lemma~\ref{lemma:partial-derivative-m-hat-zero}.
Summing up, we have
\[
H_{2,t+\delta} - H_{2,t} -
\frac{2 \alpha \delta}{\sigma^2} \iint \hat m_{t} (x) \frac{\delta^2 F}{\delta m^2} (m_t, x, y) \alpha
\bigl(\hat m_t (y) - m_t(y)\bigr)\,dx\,dy
\geq o(\delta)
\]
We have equally the bound on the other side: \(H_{2,t+\delta} - H_{2,t} \leq O(\delta^{1/p})\).

Putting everything together, we have
\begin{align*}
H_{t+\delta} - H_{t} &\leq \alpha \delta \int
\biggl(\log m_t(x) + U(x) + \frac{2}{\sigma^2} \frac{\delta F}{\delta m} (m_t, x) \biggr)
\bigl(\hat m_t(x) - m_t(x)\bigr)\,dx \\
&\quad - \frac{2\alpha\delta}{\sigma^2} \iint \frac{\delta^2 F}{\delta m^2} (m_t, x, y)
\bigl(\hat m_t (x) - m_t(x)\bigr) \bigl(\hat m_t (y) - m_t(y)\bigr)\,dx\,dy + o(\delta) \\
&=\alpha \delta \int_{\mathbb R^d} \bigl(\log m_t(x) - \log \hat m_t(x) \bigr)
\bigl(\hat m_t(x) - m_t(x)\bigr)\,dx \\
&\quad - \frac{2\alpha\delta}{\sigma^2} \iint \frac{\delta^2 F}{\delta m^2} (m_t, x, y)
\bigl(\hat m_t (x) - m_t(x)\bigr) \bigl(\hat m_t (y) - m_t(y)\bigr)\,dx\,dy + o(\delta) \\
&=- \alpha \delta \bigl( H(m_t | \hat m_t) + H(\hat m_t | m_t) \bigr) \\
&\quad - \frac{2\alpha\delta}{\sigma^2} \iint \frac{\delta^2 F}{\delta m^2} (m_t, x, y)
\bigl(\hat m_t (x) - m_t(x)\bigr) \bigl(\hat m_t (y) - m_t(y)\bigr)\,dx\,dy + o(\delta).
\end{align*}
By the convexity of \(F\).
the double integral is positive, that is to say
\[
\iint \frac{\delta^2 F}{\delta m^2} (m_t, x, y)
\bigl(\hat m_t (x) - m_t(x)\bigr) \bigl(\hat m_t (y) - m_t(y)\bigr)\,dx\,dy \geq 0.
\]
For the other side we have \(H_{t+\delta} - H_{t} \geq O(\delta^{1/p})\).
Thus \(H_t\) is continuous and \(p_t\) defined by
\begin{multline*}
p_t = - \alpha \bigl( H(m_t | \hat m_t) + H(\hat m_t | m_t) \bigr) \\
- \frac{2\alpha}{\sigma^2}
\iint \frac{\delta^2 F}{\delta m^2} (m_t, x, y) \bigl(\hat m_t (x) - m_t(x)\bigr) \bigl(\hat m_t (y) - m_t(y)\bigr)\,dx\,dy.
\end{multline*}
is an upper-differential of \(H(m_t | \hat m_t)\)
and satisfies \(p_t
\leq - \alpha \bigl( H(m_t|\hat m_t) + H(\hat m_t | m_t) \bigr)\).
\end{proof}

\begin{proof}{\bf of Theorem~\ref{thm:exp}\ }
By Proposition~\ref{prop:derivative-value-function}, we know
\[
\frac{dV^\sigma (m_t)}{dt}
= -\alpha \frac{\sigma^2}{2} \bigl( H(m_t|\hat m_t) + H(\hat m_t|m_t) \bigr).
\]
By Proposition~\ref{prop:entropy-exponential},
we find for every \(t > 0\) an upper-differential
\(p_t \in D^+ H(m_t|\hat m_t)\) such that
\[
p_t \leq -\alpha \bigl( H(m_t|\hat m_t) + H(\hat m_t|m_t) \bigr).
\]
Therefore,
\[
\frac{dV^{\sigma} (m_t)}{dt} \geq \frac{\sigma^2}{2} p_t.
\]
Since \(\frac{dV^\sigma (m_t)}{dt} - p_t\) is a lower-differential of
\(V^\sigma(m_t) - H(m_t|\hat m_t)\),
we apply Lemma~\ref{lemma:lower-differential-integration}
to the finite interval \([s, u]\) and obtain
\begin{equation}
\label{eq:final-step-expoential-convergence}
V^\sigma(m_u) - V^{\sigma}(m_s) \geq \frac{\sigma^2}{2}
\bigl( H(m_u|\hat m_u) - H(m_s | \hat m_s) \bigr).
\end{equation}
It follows from Proposition~\ref{prop:entropy-exponential}
and Lemma~\ref{lemma:lower-differential-integration}
that \(t \mapsto e^{\alpha t} H(m_t | \hat m_t)\) is non-increasing,
and therefore,
\[
H(m_t|\hat m_t) \leq H(m_0|\hat m_0) e^{-\alpha t}.
\]
Taking the limit \(u \to +\infty\)
in \eqref{eq:final-step-expoential-convergence},
we obtain
\[
\inf V^{\sigma} - V^{\sigma}(m_s)
\geq 0 - \frac{\sigma^2}{2} H(m_s|\hat m_s)
\geq - \frac{\sigma^2}{2} H(m_s|\hat m_s) e^{-\alpha t},
\]
and the proof is complete.
\end{proof}

\section{Conclusion}

In this paper we proposed the entropic fictitious play algorithm
that solves the mean field optimization problem regularized by relative entropy.
The algorithm is composed of an inner and an outer iteration,
sharing the same flavor with the particle dual average algorithm
studied in \cite{Nitanda2021ParticleDA},
but possibly allows easier implementations.
Under some general assumptions we rigorously prove the exponential convergence
for the outer iteration and identify the convergence rate
as the learning rate \(\alpha\).
The inner iteration involves sampling a Gibbs measure
and many Monte Carlo algorithms have been extensively studied for this task,
so errors from the inner iterations are not considered in this paper.
For further research directions, we may look into the discrete-time scheme
to better understand the efficiency and the bias of the algorithm,
and may also study the annealed entropic fictitious play
(i.e., \(\sigma \to 0\) when \(t \to +\infty\)) as well.

\bibliography{references}

\end{document}